\documentclass [10pt,a4paper]{article}       
\usepackage[utf8]{inputenc}
\usepackage[english]{babel}
\usepackage{lmodern}
\usepackage[square,comma,numbers]{natbib}
\usepackage{pifont,url}
\usepackage{amsmath,amssymb,amsthm,mathtools,amsfonts}
\usepackage{rotating}
\usepackage{graphicx,epstopdf,xcolor}
\usepackage{pdfpages}
\usepackage{wrapfig}
\usepackage[]{algorithm2e}
\usepackage{paralist}
\usepackage{autonum}
\usepackage{layouts}

\newtheorem{defin}{Definition}
\newtheorem{theo}[defin]{Theorem}
\newtheorem{algo}[defin]{Algorithm}

\newcommand{\dic}{\mathcal{D}}
\newcommand{\class}{\mathcal{C}}

\newcommand{\real}{\mathbb{R}}
\newcommand{\nat}{\mathbb{N}}

\newcommand{\T}{\mathcal{T}}

\newcommand{\Lp}[1]{\mathrm{L}^#1(\Omega)}
\newcommand{\SBN}{\mathcal{H}}

\DeclareMathOperator*{\argmax}{arg\,max}
\DeclareMathOperator*{\argmin}{arg\,min}
\DeclareMathOperator*{\hull}{span}
\newcommand{\lon}{\varphi}

\newcommand{\era}{\varepsilon^r}
\newcommand{\ephi}{\varepsilon^\lon}
\newcommand{\ete}{\varepsilon^t}
\newcommand{\pdervr}{\frac{\partial}{\partial r}}
\newcommand{\pdervlon}{\frac{\partial}{\partial \lon}}
\newcommand{\pdervt}{\frac{\partial}{\partial t}}

\newcommand{\ball}{\mathbb{B}}

\newcommand{\RFMP}{\mathrm{RFMP}}

\title{A first approach to learning a best basis for gravity field modelling}
\author{
	\normalsize{Volker Michel and Naomi Schneider\thanks{Corresponding author: naomi.schneider@mathematik.uni-siegen.de}}\\
	\normalsize{Geomathematics Group Siegen}}
\date{}

\begin{document}

\maketitle
\begin{abstract}
Gravitational field modelling is an important tool for inferring past and present dynamic processes of the Earth. Functions on the sphere such as the gravitational potential are usually expanded in terms of either spherical harmonics or radial basis functions (RBFs). The  (Regularized) Functional Matching Pursuit ((R)FMP) and its variants use an overcomplete dictionary of diverse trial functions to build a best basis as a sparse subset of the dictionary and compute a model, for instance, of the gravity field, in this best basis. Thus, one advantage is that the dictionary may contain spherical harmonics and RBFs. Moreover, these methods represent a possibility to obtain an approximative and stable solution of an ill-posed inverse problem, such as the downward continuation of gravitational data from the satellite orbit to the Earth's surface, but also other inverse problems in geomathematics and medical imaging. A remaining drawback is that in practice, the dictionary has to be finite and, so far, could only be chosen by rule of thumb or trial-and-error. In this paper, we develop a strategy for automatically choosing a dictionary by a novel learning approach. We utilize a non-linear constrained optimization problem to determine best-fitting RBFs (Abel--Poisson kernels). For this, we use the Ipopt software package with an HSL subroutine. Details of the algorithm are explained and first numerical results are shown.\\


\textbf{AMS:}
31B20, 
41A45,
65D15, 
65J20,
65K10,
65N20,
65R32,
68T05,
86A22
\end{abstract}

\section{Introduction} \label{sec:intro}

The gravitational potential is an important observable in the geosciences as it is used as a reference for multiple static and dynamic phenomena of the complex Earth system. The EGM2008 gives us a high-precision model in spherical harmonics, i.e. polynomials, up to degree 2190 and order 2159, see \citep{EGM2008,Pavlisetal2012}. From satellite missions like GRACE or its successor GRACE-FO, we have time-dependent models of the potential, see, for example, \citep{Flechtneretal2014,GRACEdata,Schmidtetal2008,Tapleyetal2004}. These data enable a visualization of mass transports on the Earth such as seasonal short-term phenomena like the wet season in the Amazon basin as well as long-term phenomena like the climate change. Therefore, gravity field modelling and especially the downward continuation of satellite data is one of the major important mathematical problems in physical geodesy, see, for instance, \citep{Baur2014,Kusche2010}.

From a mathematical point of view, the gravitational potential $F$ on the approximately spherical Earth's surface can be modelled as a Fourier expansion in a suitable basis, for example in the mentioned spherical harmonics $Y_{n,j},\ n \in \nat_0,\ j=-n,...,n$. If we assume the Earth to be a closed unit ball, we obtain, for $\sigma>1$, a pointwise representation of the potential as 
\begin{align}
V(\sigma\eta) = (\T F) (\sigma\eta) = \sum_{n=0}^\infty \sum_{j=-n}^n \langle F, Y_{n,j} \rangle_{\Lp{2}} \sigma^{-n-1} Y_{n,j} \left( \eta \right) \label{eqpotential}
\end{align}	
for the unit sphere $\Omega$ and $\eta \in \Omega$, see, for example \citep{Baur2014,Freedenetal2004,Moritz2010,Telschow2014}. This gives us the potential in the outer space including a satellite orbit. The inverse problem of the downward continuation of this potential is given as follows: if data values $V(\sigma\eta) = (\T F) (\sigma\eta),\ \sigma>1$, are known, determine the function $F$ on $\Omega$. For more details on inverse problems in general, see the classical literature, for example, \citep{Engletal1996,Louis1989,Rieder2003}. The occurring mathematical challenges of the downward continuation are well-known. First of all, the operator $\T$ has exponentially decreasing singular values due to $\sigma>1$ in \eqref{eqpotential}. Thus, the inverse operator which we need for the downward continuation has exponentially increasing singular values. For this reason, the inverse problem is called exponentially ill-posed. In particular, it violates the third characteristic of a well-posed problem according to Hadamard (continuous dependence on the data). Furthermore, the existence of $F$ is only ensured if $V$ is in the range of $\T$. However, if $F$ exists, then it is unique.

Therefore, sophisticated algorithms need to be used to solve the problem of the downward continuation of satellite data. Previous studies showed that the (Regularized) Functional Matching Pursuit ((R)FMP), the (Regularized) Orthogonal Functional Matching Pursuit ((R)OFMP) as well as the latest (Regularized) Weak Functional Matching Pursuit ((R)WFMP) are possible approaches for this and other inverse problems, see, for instance, \citep{Berkeletal2011,Fischer2011,Fischeretal2012,Fischeretal2013_1,Fischeretal2013_2,Guttingetal2017,Kontak2018,Kontaketal2018_2,Kontaketal2018_1,Michel2015_2,Micheletal2017_1,Micheletal2014,Micheletal2016_1,Telschow2014}. In the sequel, we will write Inverse Problem Matching Pursuit (IPMP) if we refer to either one of the mentioned algorithms. Although the core routine of these algorithms is well established by now, there are still possibilities to improve their performance.

One of these possibilities is given due to the following circumstances. The IPMPs are based on a finite dictionary $\dic$ of suitable trial functions from which they build a best basis and eventually the approximate solution in terms of this best basis. Originally, matching pursuits utilize a dictionary consisting of vectors from finite-dimensional spaces. The first development of a matching pursuit was done by S.G. Mallat and Z. Zhang (\citeyear{Mallatetal1993}). The ROFMP is additionally based on works of P. Vincent and Y. Bengio (\citeyear{Vincentetal2002}) and Y.C. Pati et al. (\citeyear{Patietal1993}). The RWFMP inherits ideas from V.N. Temlyakov (\citeyear{Temlyakov2000}). The idea of using dictionaries instead of finding a representation of a signal in an a-priori given basis can be summed up as follows \citep{Mallatetal1993}: the human language gives us nearly infinite possibilities to describe the very same thing in the real world. However, these descriptions vary in length and rigour. This idea can be transferred to mathematics, for instance, gravity field modelling. We can model the gravitational potential in spherical harmonics as given in \eqref{eqpotential}. However, if we look for a model in a best basis from a dictionary $\dic$, the representation of the signal might be sparser and/or more precise. In particular, the reduction to those basis functions which are essential increases the interpretability of the obtained model. Further, numerical experiments showed that the obtained solution is more accurate and stable. These aims can be achieved by the IPMPs. 

Further characteristics of the IPMPs can be summed up from previous publications as follows: they represent a regularization for ill-posed inverse problems; they can combine different kinds of trial functions, e.g. global and localized ones, in their solution; 
they can be used for pure interpolation / approximation as well as for linear and non-linear ill-posed inverse problems; they work with single-source data as well as are capable of a joint-inversion of multiple-source data; the data can be given on different geometries, like a sphere, a ball, or an interval of the real line; they yield an approximative function and not a discretized approximation; they build this solution iteratively without the need to invert a matrix or solve a large linear system of equations; the orthogonal variant yields a linear combination of orthogonal trial functions; the runtime can be improved with the use of preprocessing of certain data; or they can be combined with a weak strategy to cut runtime without significant loss of accuracy; the implementation is easy and they can be parallelized very well. 

In the practical tests for diverse applications, very good approximations could be achieved not only for the downward continuation but also for other ill-posed inverse problems, for instance the (linear as well as the non-linear) inverse gravimetric problem or the inversion of MEG- and EEG-data, see, for instance, \citep{Fischer2011,Fischeretal2012,Kontak2018,Kontaketal2018_2,Leweke2018,Lewekeetal2018_3}. 

However, the experiments also revealed a sensitivity of the result regarding the choice of the dictionary, for example concerning the runtime and the convergence behaviour. Therefore, the main focus of this paper is on a first dictionary learning strategy for the downward continuation of gravitational potential data.

Previous works on dictionary learning considered discretized approximation problems. In this case, the dictionary can be interpreted as a matrix. The approaches aimed to obtain a solution of the approximation problem and a sparse dictionary matrix simultaneously. For more details, see, for example, \citep{Brucksteinetal2009,Pruente2008,Rubinsteinetal2010}.  

However, a particular feature of the IPMPs is that their solution is a linear combination of established trial functions. Neither do we want to discretize the dictionary elements, i.e. the trial functions, nor do we want to modify them. In the latter case, the comparability with traditional models in these trial functions would be lost. Furthermore, with the use of scaling functions and / or wavelets in the dictionary, the IPMP generates a solution in a multiscale basis. This allows a multiresolution analysis of the obtained model revealing hidden local detail structures as it was shown in, for example, \citep{Fischer2011,Fischeretal2012,Fischeretal2013_1,Fischeretal2013_2,Micheletal2014,Telschow2014}. Moreover, we do not only consider interpolation / approximation problems, but also ill-posed inverse problems. Thus, a dictionary matrix would not contain the basis elements themselves, but, for example, their upward continued values. Applying previous strategies, like, for instance, MOD or K-SVD, would only alter the upward continued values and leave us with the question of how to downward continue them. All in all, this shows that learning a dictionary for the IPMPs requires the development of a different strategy.

For a first approach to learning a dictionary, we concentrate on the RFMP as the basic IPMP in this paper. For this algorithm, we develop a procedure to determine a best basis for the gravitational potential from different types of infinitely many trial functions. We choose to learn dictionary elements from spherical harmonics and Abel--Poisson kernels as radial basis functions. In particular, while previously a discrete grid of centres of the RBFs had to be chosen a-priori, which could have put a bias in the obtained numerical result, we now allow every point on the unit sphere to be a centre of an RBF. Equally, the localization parameter of the Abel--Poisson kernel is now determined from an interval instead of a finite set. Our continuous, i.e. non-discrete, learning ansatz produces a 'best-dictionary' with which the RFMP can be run. We call this procedure the Learning (Regularized) Functional Matching Pursuit (L(R)FMP). The results show that the use of a learnt dictionary in the RFMP gives us a higher sparsity and better results with less storage demand. 

This paper is structured as follows. On the way to a detailed description of our learning approach, we define some fundamental basics in Section \ref{sec:basics}. We introduce the trial functions under investigation as well as a general form of a dictionary and the basic principles of the RFMP. With these aspects explained, we state our learning strategy in Section \ref{sec:learning}. We motivate its idea and explain how this is embedded into the established theory of dictionary learning. Then we define its routine and give necessary derivatives. These are derived in Appendix \ref{app}. We end Section \ref{sec:learning} by introducing some additional learning features that guide the learning process positively in practice. In Section \ref{sec:numerics}, we describe experiments for which we learn a dictionary and compare the results of the learnt dictionary with the results which a manually chosen dictionary yields. At last, we conclude this paper in Section \ref{sec:conandout} with an outlook of how we want to further develop this first learning approach.

\section{Some mathematical tools for learning a dictionary} \label{sec:basics}

\subsection{Trial functions under consideration for dictionaries} \label{subsec:trialfun}

First of all, to develop a learning strategy, we have to define what trial functions we want to consider as possible dictionary elements, i.e. what the learnt dictionary shall consist of. Our aim is to determine which well-known trial functions are most suitable for a problem at hand. For gravity field modelling, it is sensible to determine suitable spherical harmonics as well as Abel--Poisson kernels. Examples of those trial functions are given in Figure \ref{fig:functions}.

\begin{figure}[htbp]
	\includegraphics[trim= 16mm 31mm 11mm 23mm, clip,width=.49\linewidth]{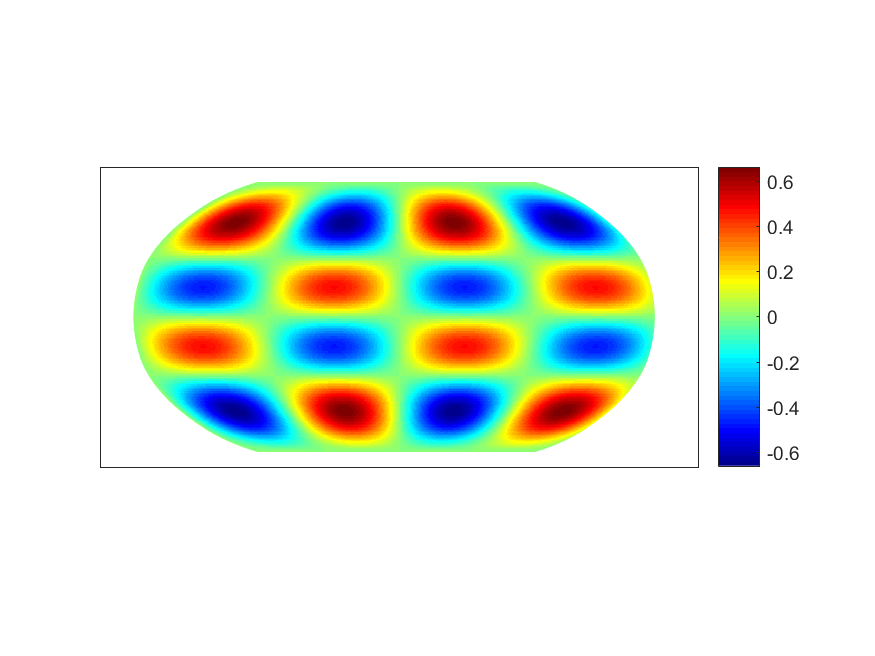}
	\includegraphics[trim= 16mm 31mm 11mm 23mm, clip,width=.49\linewidth]{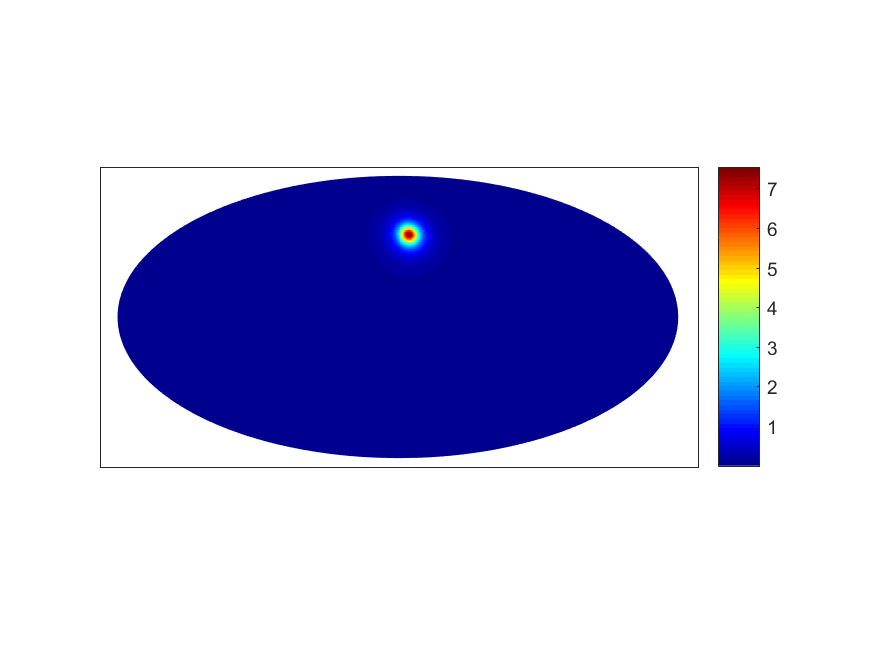}
	\caption{Examples of trial functions. Left: spherical harmonic. Right: Abel--Poisson kernel for a fixed centre $\xi$ (location of the maximum).} \label{fig:functions}
\end{figure}

Spherical harmonics or fully normalized spherical harmonics for practical purposes are global trial functions, see for instance, \citep{Freedenetal2013_1,Freedenetal2009,Freedenetal1998,Michel2013,Mueller1966}. 
An example is given on the left-hand side of Figure \ref{fig:functions}. They are defined for a unit vector $\xi \in \Omega$ as 
\begin{align}
Y_{n,j} \left( \xi(\lon,t) \right) \coloneqq \sqrt{\frac{(2n+1)}{4\pi} \frac{(n-|j|)!}{(n+|j|)!}} P_{n,|j|}(t)  \left\{ \begin{matrix}
\sqrt{2} \cos(j\lon), &j<0,\\
1, &j=0,\\
\sqrt{2} \sin(j\lon), &j>0,
\end{matrix} \right. \label{eqfnsh}
\end{align}
where $\xi(\lon,t)$ is the representation of $\xi \in \Omega$ in polar coordinates $(\lon,t)$, where $t=\cos \vartheta$ and $\vartheta$ is the latitude. Further, the definition uses associated Legendre functions given by
\begin{align}
P_{n,j}(t) \coloneqq \left( 1-t^2 \right)^{j/2} \frac{\mathrm{d}^j}{\mathrm{d}t^j} P_n(t), \qquad t \in [-1,1], \label{eqaLf}
\end{align}
where $P_n$ denotes the $n$-th Legendre polynomial. 

Abel--Poisson kernels are defined for a particular unit vector $\xi \in \Omega$ and a scaling parameter $h \in [0,1)$  as (with $x=h\xi$)
\begin{align}
P(x,\eta) \coloneqq \frac{1-|x|^2}{4\pi(1+|x|^2 - 2x \cdot \eta)^{3/2}} 
= \sum_{n=0}^\infty \frac{2n+1}{4\pi} |x|^n P_n\left( \frac{x}{|x|} \cdot \eta \right) \label{eqAPk}
\end{align}
for any unit vector $\eta \in \Omega$, see, for example, \citep[][pp.~108-112]{Freedenetal1998} or \citep[][p.~103~and~441]{Freedenetal2009}. These kernels are more localized than polynomials as one can see on the right-hand side of Figure \ref{fig:functions}. It is visible that they are radial basis functions, that means they have one maximum whose descent depends on the distance to the centre $\xi = x/|x|$ of the extremum. In that way, they are zonal functions and can be viewed as 'hat'-functions. Dependent on the parameter $h = |x|$, the size of the extremum or 'hat' varies in size. Thus, the functions have different scales of localization. For more details and examples, see, for instance, \citep[][p.~111]{Freedenetal1998} or \citep[][p.~117]{Michel2013}.

In this paper, we consider dictionaries consisting of spherical harmonics and Abel--Poisson kernels. We introduce here a notation for building blocks of spherical dictionaries.

\begin{defin} \label{def:dic}
Let $N \subset \mathcal{N} \coloneqq \{(n,j)\ |\ n \in \nat_0, j=-n,...,n\}$ and $ K \subseteq \mathring{\ball}_1(0)$ for the open unit ball $\mathring{\ball}_1(0)$. Then we set 
\begin{align}
[N]_\mathrm{SH} &\coloneqq \{Y_{n,j}\ |\ (n,j) \in N \}
\intertext{for spherical harmonics $Y_{n,j}$ and}
[K]_\mathrm{APK} &\coloneqq \{P(r\xi,\cdot)\ |\ r\xi \in K\}
\intertext{for Abel--Poisson kernels $P(r\xi,\cdot)$. We define a dictionary as}
\dic &\coloneqq [N]_\mathrm{SH} + [K]_\mathrm{APK} \coloneqq [N]_\mathrm{SH} \cup [K]_\mathrm{APK}.
\end{align}
We call $[\cdot]_\ast$ a trial function class.
\end{defin}
Note that $N$ and $K$ may be finite or infinite. 

\subsection{Basic principles of linear ill-posed inverse problems} \label{sebsec:iiprobs}

This subsection is mainly based on \citep{Engletal1996,Kontak2018,Michel2013,Rieder2003}. First of all, we recall the definition of a linear inverse problem.

\begin{defin}\label{def:invprob}
Let $\mathcal{X},\ \mathcal{Y}$ be Hilbert spaces. Further, let $\T \colon \mathcal{X} \to \mathcal{Y}$ be a linear and continuous operator between them. At last, let $y \in \mathcal{Y}$ denote the data and $F \in \mathcal{X}$ the desired solution. Then a problem of the form 
\begin{align}
\T F = y \label{eqinvprob}
\end{align}
is called a linear inverse problem.
\end{defin}

Naturally, some fundamental mathematical questions about the problem arise. Those questions lead to the definition of well- and ill-posedness of the linear inverse problem by Hadamard.

\begin{defin} \label{def:Hadamard}
A linear inverse problem $\T F = y$ is called well-posed if it fulfils the following three properties.
\begin{compactitem}
\item[(a)] For each $y$, there exists a solution $F$.
\item[(b)] The solution $F$ is unique.
\item[(c)] The inverse operator $\T^{-1}$ is continuous, i.e. the solution $F$ is stable.
\end{compactitem}
If any of these properties is violated, the problem is called ill-posed.
\end{defin}

As we explained in the introduction, the problem of the downward continuation of satellite data for gravity field modelling is a linear ill-posed inverse problem. As most inverse problems from practice are ill-posed, there exists a large theory on how to still solve these problems. In the sequel, we will use the approach by Tikhonov.

Its idea is to find the best approximate solution instead of the true solution. The ill-posedness is treated with an additional penalty term. The best approximate solution is the minimizer of the Tikhonov functional for a vanishing regularization parameter. The functional is given as follows. 

\begin{defin} \label{def:Tikhgen}
Let $\T F = y$ be a linear inverse problem. The Tikhonov functional $\mathcal{J} \colon \mathcal{X} \to \real$ is given by 
\begin{align}
\mathcal{J} (F;\T,  \lambda,y) \coloneqq \| y - \T F \|^2_{\mathcal{Y}} + \lambda \| F \|^2_{\mathcal{X}}. \label{eqTikhgen}
\end{align}
\end{defin}

In the case of satellite data, we set $\mathcal{Y} = \real^\ell$, i.e. we are only given data on $\ell$ discrete points of the unit sphere. For the space $\mathcal{X}$, we propose to use a Sobolev space. The reasons for this choice are, for instance, that this space enforces more smoothness of the solution than, e.g. the $\Lp{2}$-space. This has proven to yield better results. Further, it was shown that the IPMPs also profit theoretically from the use of this space, see, \citep{Kontak2018}. Specifically, we will use the Sobolev space $\SBN_2$.

\begin{defin} \label{def:H2}
On the set $\widetilde{\SBN}\left( (n+0.5)^2; \Omega \right)$ of all functions $F \in \mathrm{C}^{(\infty)}(\Omega,\real)$ that fulfil
\begin{align}
\sum_{n=0}^{\infty} \sum_{j=-n}^{n} (n+0.5)^4 \langle F, Y_{n,j} \rangle_{\Lp{2}}^2 < \infty,
\label{eqH2}
\end{align}
we define an inner product via 
\begin{align}
\langle F, G \rangle_{\SBN_2} \coloneqq \sum_{n=0}^{\infty} \sum_{j=-n}^{n} (n+0.5)^4 \langle F, Y_{n,j} \rangle_{\Lp{2}} \langle G, Y_{n,j} \rangle_{\Lp{2}}.
\label{eqH2IP}
\end{align}
The completion of $\widetilde{\SBN}\left( (n+0.5)^2; \Omega \right)$ with respect to $\langle \cdot, \cdot \rangle_{\SBN_2}$ is called the Sobolev space $\SBN_2$.
\end{defin}

For practical purposes, we will give a short overview of the main principles of the RFMP algorithm, which is a regularization method for ill-posed linear problems, next. For more details on any IPMP, see, for example, \citep{Fischer2011,Fischeretal2012,Fischeretal2013_1,Fischeretal2013_2,Kontak2018,Kontaketal2018_1,Kontaketal2018_2,Micheletal2017_1,Micheletal2014,Micheletal2016_1,Telschow2014}. For theoretical discussions of the IPMPs, we refer to this literature. Here, we will concentrate on the practical aspects of the RFMP. 

The underlying idea of this matching pursuit is to build a solution as a linear combination of dictionary elements by iteratively minimizing a Tikhonov functional. In theory, we consider the linear inverse problem $\T F = V$, for instance, as given in \eqref{eqpotential}. In practice, we consider the particular case $\mathcal{Y} = \real^\ell$. Then the problem formulates as follows. We have a relative satellite height $\sigma>1$, a set of grid points $\{\eta^{(i)}\}_{i=1,...,\ell} \in \Omega$ and data values $y_i$ for each grid point $\eta^{(i)}$. The operator $\T$ is exchanged by a finite system of related functionals $\T^i_\daleth$ for which $\T^i_\daleth F = (\T F)(\sigma\eta^{(i)}) = y_i$ holds for $i=1,...,\ell$. We use the Hebrew letter Dalet $\daleth$ to emphasize that the functionals $\left( \T^i_\daleth \right)_{i=1,...,\ell}$ represent a discretization of the operator $\T$. Summarized, we consider the linear inverse problem $\T_\daleth F = y$ for the operator $\T_\daleth = (\T^i_\daleth)_{i=1,...,\ell}$ and $y \in \real^\ell$.

Further, we assume that $F \in \SBN_2$. A regularization parameter is denoted by $\lambda$. Additionally, we need an a-priori defined dictionary $\dic$ as given in Definition \ref{def:dic}. Then the aim of the RFMP is to iteratively minimize the Tikhonov functional 
\begin{align}
\mathcal{J}(F_n+\alpha d; \T_\daleth, \lambda, y) \coloneqq \| y - \T_\daleth (F_n+\alpha d)\|^2_{\real^\ell} + \lambda \|F_n + \alpha d\|^2_{\SBN_2} \label{eqTikh}
\end{align}
for an element $d \in \dic$ of the dictionary, a real coefficient $\alpha$ and a current approximation $F_n$. In practice, this means we start with an initial approximation $F_0$, e.g. $F_0 \equiv 0$, and iteratively determine $F_{n+1} \coloneqq F_n + \alpha_{n+1}d_{n+1}$ via
\begin{align}
(\alpha_{n+1}, d_{n+1} ) \coloneqq \argmin_{(\alpha,d) \in \real \times \dic} \mathcal{J}(F_n+\alpha d; \T_\daleth, \lambda, y). \label{eqargmin}
\end{align}
It can be shown, see, for example, \citep{Fischer2011,Michel2015_2,Micheletal2014} that the minimization with respect to $\alpha$ and $d$ of the Tikhonov functional \eqref{eqargmin} is equivalent to a maximization of a certain quotient with respect to $d$. For the RFMP, this quotient is given by
\begin{align}
d_{n+1} \coloneqq \argmax_{d \in \dic} \frac{ \left( \left\langle R^n, \T_\daleth d \right\rangle_{\real^\ell} - \lambda \left\langle F_n , d \right\rangle_{\SBN_2} \right)^2}{\| \T_\daleth d \|^2_{\real^\ell} + \lambda \| d \|^2_{\SBN_2} }. \label{eqmaxRFMP}
\end{align}
in the $n$-th iteration step, where $R^n \coloneqq y - \T_\daleth F_n$ is the residual. 

\section{The learning approach} \label{sec:learning}

In this section, we refer to a linear inverse problem $\T_\daleth F = y$ for $y \in \real^\ell$ as described above. The term $F_n$ represents the current approximation of the RFMP at iteration step $n$. A dictionary element is denoted by $d$, spherical harmonics by $Y_{n,j}$ and Abel--Poisson kernels by $P(x,\cdot)$ as we introduced them in the last section. 

\subsection{About learning dictionaries for inverse problems} \label{subsec:learndic}

We motivate our learning algorithm as follows. We consider an infinite set of trial functions from which we want to learn dictionary elements in the LRFMP. We set 
\begin{align}
\dic^{\mathrm{inf}} = [\mathcal{N}]_\mathrm{SH} + \left[\mathring{\ball}_1(0)\right]_\mathrm{APK}.
\end{align}

Thus, for $F_n = \sum_{i=1}^n \alpha_i d_i$, we consider 
\begin{align}
\argmin_{(\alpha_{n+1},d_{n+1}) \in \real \times \dic^{\mathrm{inf}}}\ \left( \| y - \T_\daleth (F_n + \alpha d) \|^2_{\real^\ell} + \lambda \| F_n + \alpha d\|^2_{\SBN_2} \right). \label{eqminLRFMP-TF-1}
\end{align}

If we chose the dictionary element $d_i$ and its coefficient $\alpha_i$ in this greedy manner in each iteration step, in a perfect world and for $n \to \infty$, then this would be equal to 
\begin{align}
\min_{F_\infty \in \ \overline{\hull \dic^{\mathrm{inf}}}} \left( \| y - \T_\daleth F_\infty \|^2_{\real^\ell} + \lambda \| F_\infty \|^2_{\SBN_2} \right), \qquad F_\infty = \sum_{i=1}^\infty \alpha_i d_i. \label{eqminLRFMP-TF-2}
\end{align}
Unfortunately, this is practically impossible to solve for several reasons. First of all, the problem is of the type of the travelling salesman problem which is known to be NP-hard, see, for instance, \citep[p.~114]{Gareyetal2009}. Thus, we simply cannot be sure that picking trial functions in a greedy manner also yields a true greedy algorithm and, thus, the optimal solution $F_\infty$. Secondly, in practice, we can only compute a finite linear combination, i.e. an optimal best-$n$-term approximation if at all. And thirdly, if we work with infinitely many trial functions, we cannot preprocess any data. However, it is an expensive feature to compute everything needed on the fly.

However, if we consider the RFMP, we know that its solution is a good approximation of 
\begin{align}
\min_{F_n \in \ \hull \dic} \left( \| y - \T_\daleth F_n \|^2_{\real^\ell} + \lambda \| F_n \|^2_{\SBN_2} \right), \qquad F_n = \sum_{i=1}^n \alpha_i d_i \label{eqminRFMP}
\end{align}
for a finite dictionary $\dic$. Note that the approximation is also a $n$-term approximation. That means, the structure of the RFMP yields a good approximation of the best-$n$-term solution dependent on a fixed finite dictionary. Therefore, if we extend the RFMP to the infinite set of functions $\dic^\mathrm{inf}$ in the way defined in \eqref{eqminLRFMP-TF-1}, we are able to approximate a solution $F_n$ of
\begin{align}
\min_{F_n \in \ \hull \dic^\mathrm{inf}} \left( \| y - \T_\daleth F_n \|^2_{\real^\ell} + \lambda \| F_n \|^2_{\SBN_2} \right),\qquad F_n = \sum_{i=1}^n \alpha_i d_i. \label{eqminLRFMP-TF-3}
\end{align}
This is the main objective of the LRFMP in practice. How can we learn a finite dictionary $\dic$ from this solution? If we have approximated the solution by $F_n$, we know which basis elements we need for this approximation. In this way, we know which dictionary elements should be at least in the learnt finite dictionary $\dic$ for its use in the RFMP.  

Therefore, a learnt dictionary should contain at least these $n$ elements. More generally, we can set an upper bound $D > n$ for the size of the learnt dictionary $\dic$. Furthermore, the approximation of $F_n$ should be in the span of the learnt dictionary $\dic$, such that the solution of \eqref{eqminLRFMP-TF-3} can be reproduced. Moreover, the dictionary $\dic$ naturally needs to be a subset of $\dic^\mathrm{inf}$. Taking all things into consideration, we can write the objective of the dictionary learning process as 
\begin{align}
\min_{\substack{\dic \subset \dic^\mathrm{inf},\\ |\dic| \leq D}}\ \min_{F_n \in \ \hull \dic} \left( \| y - \T_\daleth F_n \|^2_{\real^\ell} + \lambda \| F_n \|^2_{\SBN_2} \right). \label{eqminLIPMP-TF-4}
\end{align}	
This is a doubled minimization problem for a predefined dictionary size $D$. From another point of view, the task it presents is to find a dictionary and build an approximation from its elements that minimizes the Tikhonov functional simultaneously. A doubled minimization problem is a common approach in the field of dictionary learning, see, for instance, \citep{Aharonetal2006,Brucksteinetal2009,Pruente2008,Rubinsteinetal2010}.
Further, referring to the criteria for dictionary learning from Aharon et. al. (\citeyear{Aharonetal2006}), this can be thought of as a well-defined goal of the learning approach. 

Thus, the starting point for our learning algorithm is \eqref{eqminLRFMP-TF-1}. We choose this generalization for the learning approach because we want to learn the dictionary itself as well as maintain the unique characteristic of the RFMP to combine various but established trial functions. Moreover, for this minimization, we choose to follow the structure of the RFMP for two reasons. First of all, as we pointed out above, the RFMP is built to solve the same minimization problem but over a finite set of functions. Secondly, we search for a well-working dictionary for the RFMP. It would be intuitive to expect better results when applying the learnt dictionary if we included the behaviour of the RFMP in the learning process. In this way, we use as much information and structure as we have to obtain an efficient learning routine. Hence, the idea of our learning algorithm, the LRFMP, is to iteratively minimize a Tikhonov functional over an infinite set of trial functions similarly to the RFMP in order to model the solution in a best basis. The chosen basis elements form the learnt dictionary which can be applied in runs of the RFMP.

\newpage
\subsection{A first learning algorithm} \label{subsec:learningstruct}

\begin{wrapfigure}{l}{.6\linewidth}
\includegraphics[scale=0.6]{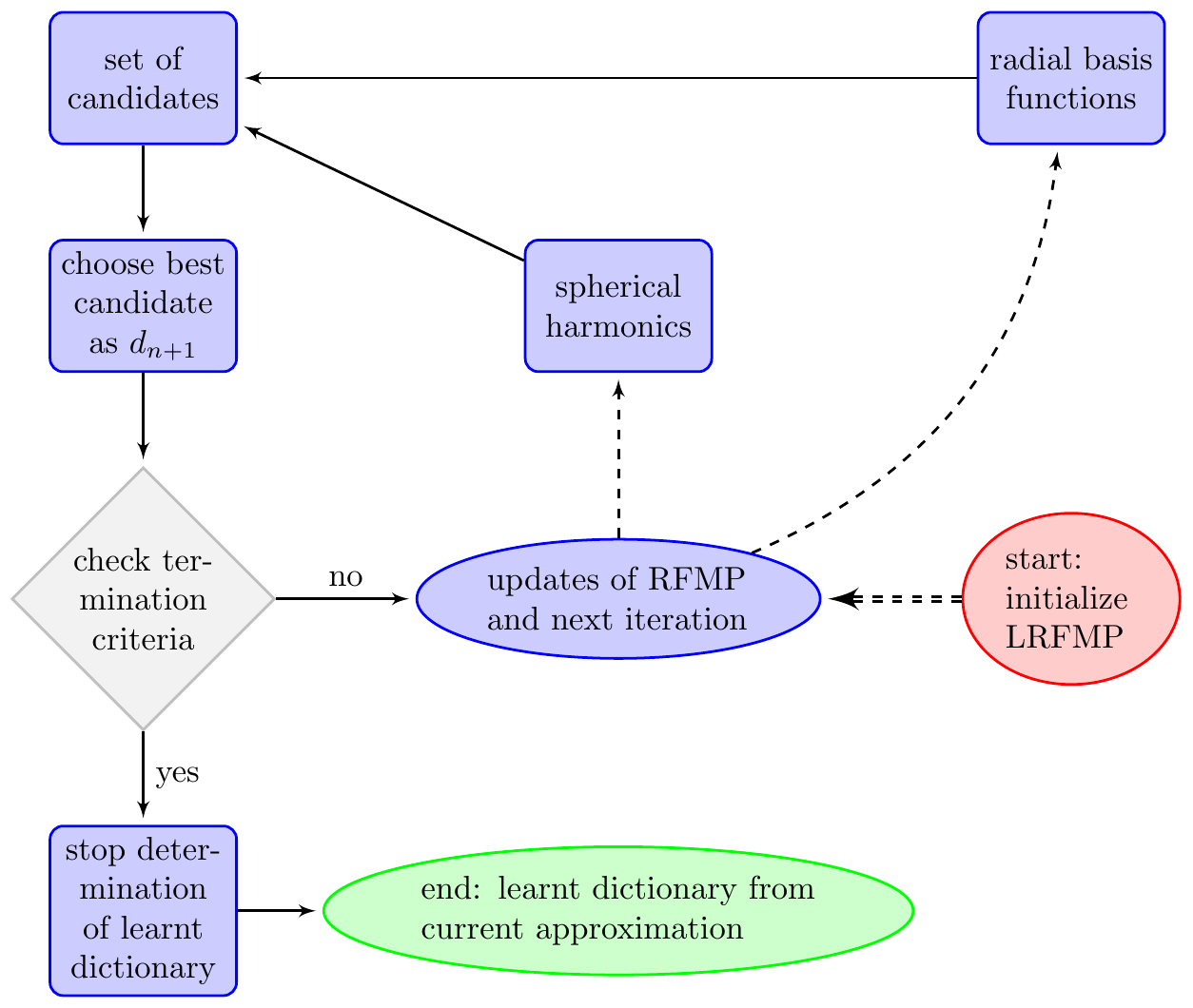}
\caption{Schematic representation of the basic learning algorithm.}
\label{fig:structLIPMP}
\end{wrapfigure}

An overview of the structure of the learning algorithm is shown in Figure \ref{fig:structLIPMP}. We start in the red circle ('start') where we initialize the LRFMP similarly to the initialization which the RFMP needs. This means, the initialization includes the necessary preprocessing and setting of parameters (similarly as described, for example, in \citep{Telschow2014}) as well as setting parameters for the learning. The latter learning parameters include, most importantly, a starting dictionary and smoothing properties. 

Then we step into the first iteration in which we want to minimize the Tikhonov func\-tion\-al in order to find $d_1$. As we also want to learn a dictionary, the steps up to choosing $d_1$ differ from the established RFMP: we choose $d_1$ from (in the case of the RBFs) infinitely many trial functions instead of a finite a-priori selection of trial functions. This is done by first computing a candidate for $d_1$ among each trial function class we consider. In Figure \ref{fig:structLIPMP}, this is shown by the boxes 'spherical harmonics' and 'radial basis functions' which lead to 'set of candidates'. Mathematically, in this step, we seek 
\begin{align}
&\left( \alpha^{\class}_{1}, d^{\class}_{1} \right) = \argmin_{(\alpha,d)\ \in \ \class \times \real} \left( \| y - \T_\daleth (F_0 + \alpha d) \|^2_{\real^\ell}  + \lambda \| (F_0 + \alpha d) \|^2_{\SBN_2} \right), \label{eqminLIPMPclasses-1}
\end{align}
where $\class$ denotes one trial function class, i.e. $\class = [\mathcal{N}]_\mathrm{SH}$ the set of spherical harmonics or $\class = [\mathring{\ball}_1(0)]_\mathrm{APK}$ the set of Abel--Poisson kernels. Then we have again a finite (but optimized) set of trial functions and can choose $d_1$ from this set of candidates in the common fashion of the RFMP by comparing how well each one minimizes the Tikhonov functional. The candidate that minimizes the Tikhonov functional among all candidates is chosen as $d_1$ ('choose best candidate as $d_{n+1}$').

Then we compute the necessary updates of the RFMP-routine as described, for example, in \citep{Telschow2014}. Next, we check the termination criteria for the learning algorithm. We adopt the termination criteria of the RFMP which are up to now the norm of the residual or the size of the currently chosen coefficient (this would be $\alpha_1$ at this stage) being smaller than a given threshold or a maximal number of iterations. Either they are not fulfilled, then we search for the next element $d_2$ in the same manner as we found $d_1$. Or we have a fulfilled termination criterion. In this case, we stop the RFMP and, thus, the learning of the dictionary. We obtain the same output as in a non-learning RFMP which is an approximation of the given signal. Additionally, the learnt dictionary is defined as the set of all chosen elements in this approximation. For the sake of completeness, note that for an arbitrary iteration step $n$, the objective to seek $\alpha_{n+1}$ and $d_{n+1}$ is given by 
\begin{align}
&\left( \alpha^{\class}_{n+1}, d^{\class}_{n+1} \right) = \argmin_{(\alpha, d)\ \in \ \real \times \class} \left( \| y - \T_\daleth (F_n + \alpha d) \|^2_{\real^\ell}  + \lambda \| (F_n + \alpha d) \|^2_{\SBN_2} \right), \label{eqminLIPMPclasses-n}
\intertext{for}
&\class = [\mathcal{N}]_{\mathrm{SH}} \textrm{ and } \class = \left[ \mathring{\ball}_1(0)\right]_{\mathrm{APK}}, 
\end{align}
respectively.

Before we can state the learning algorithm itself, we first have to define an objective function for the determination of the candidates. 

\begin{defin} \label{def:objfunRFMP}
For the Sobolev space $\SBN_2$, we define the objective function of the RFMP in the $n$-th iteration step as 
\begin{align}
\RFMP (d;n) &\coloneqq \frac{ \left( \left\langle R^n, \T_\daleth d \right\rangle_{\real^\ell} - \lambda \left\langle F_n , d \right\rangle_{\SBN_2} \right)^2}{\| \T_\daleth d \|^2_{\real^\ell} + \lambda \| d \|^2_{\SBN_2} }, \label{eqobjfunRFMP}
\end{align}
where $d$ is a trial function, $R^n$ is the current residual, $F_n$ the current approximation and $\T, \lambda$ depend on the linear inverse problem.
\end{defin}

\begin{theo} \label{th:equivminmax}
The minimization of the Tikhonov functional in the $n$-th step of the RFMP with respect to a trial function $d$ and a real coefficient $\alpha$ as seen in \eqref{eqTikh} is equivalent to the maximization of $\RFMP(\cdot;n)$ with respect to a trial function $d$.
\end{theo}
\begin{proof}
The proof is analogous to the corresponding proofs in \citep{Fischer2011,Michel2015_2}.
\end{proof}

All in all, we can now state the learning algorithm. We will explain the steps of this algorithm in further detail below. Note that we determine a preliminarily optimal Abel--Poisson kernel from a discretely parametrized dictionary $[ \hat{K} ]_{\mathrm{APK}}$ and use this as a starting point for the optimization procedure for a continuously parametrized ansatz which uses $[ \mathring{\ball}_1(0)]_{\mathrm{APK}}$.

\begin{algo} \label{algo:LRFMPmain}
We obtain a learnt dictionary for the RFMP as follows. Let $\T_\daleth F = y$ be the linear inverse problem under investigation, $\SBN_2$ the Sobolev space from Definition \ref{def:H2} and $\lambda$ the regularization parameter.\\

\begin{compactitem}
\item[(S0)] initialize: at least one termination criterion (maximal number of iterations $I$ and/or prescribed accuracy of the relative data error and/or minimal size of the coefficients); data vector $y$; initial approximation $F_0$; sets $\hat{N}$ as well as $\hat{K}$ as in Definition \ref{def:dic} and starting dictionary $\dic = [\hat{N}]_\mathrm{SH} + [\hat{K}]_\mathrm{APK}$
\item[(S1)] set $R^0 \coloneqq y-\T_\daleth F_0$ and compute $\langle F_0, d_i \rangle_{\SBN_2}$ for $d_i \in \dic$.
\item[(S2)] compute $\T_\daleth d$ for each $d \in \dic$ evaluated at the data points of $y$ and $\langle d_i, d_j \rangle_{\SBN_2}$ for each two $d_i,\ d_j \in \dic$.
\item[(S3)] while (termination not fulfilled)
	\begin{compactitem}
	\item[(S3.1)] compute candidate $$d_{n+1}^{\mathrm{SH}} = \argmax \left\{\RFMP(d;n)\ \Big|\ d \in \left[\hat{N}\right]_\mathrm{SH}\right\}$$
	\item[(S3.2)] compute local solution $$d_{n+1}^{\mathrm{APK,loc}} = \argmax \left\{\RFMP(d;n)\ \Big|\ d \in \left[\hat{K}\right]_\mathrm{APK}\right\}$$
	\item[(S3.3)] compute optimal solution $$d_{n+1}^{\mathrm{APK}} = \argmax \left\{\RFMP(d;n)\ \Big|\ d \in \left[\mathring{\ball}_1(0)\right]_\mathrm{APK}\right\}$$
	\item[(S3.4)] choose 
	\begin{align}	
	d_{n+1} = \argmax &\left\{ \RFMP\left(d_{n+1}^{\mathrm{SH}};n \right),\ \right. \\
	& \qquad  \left. \RFMP\left(d_{n+1}^{\mathrm{APK,loc}};n\right),\ \RFMP\left(d_{n+1}^{\mathrm{APK}};n\right) \right\}
	\end{align}
	\item[(S3.5)] compute $$\alpha_{n+1} = \frac{ \left\langle R^n, \T_\daleth d_{n+1} \right\rangle_{\real^\ell} - \lambda \left\langle F_n , d_{n+1} \right\rangle_{\SBN_2} }{\| \T_\daleth d_{n+1} \|^2_{\real^\ell} + \lambda \| d_{n+1} \|^2_{\SBN_2} }$$
	\item[(S3.6)] set $R^{n+1} \coloneqq R^n - \alpha_{n+1}\T_\daleth d_{n+1}$
	\item[(S3.7)] for $d_i \in \dic$ compute $\langle F_{n+1}, d_i \rangle_{\SBN_2} = \langle F_n, d_i \rangle_{\SBN_2} + \alpha_{n+1}\langle d_{n+1}, d_i \rangle_{\SBN_2}$
	\item[(S3.8)] increase $n$ by $1$
	\end{compactitem}
\item[(S4)] result: approximation $F_M = \sum_{i=1}^M \alpha_i d_i$ after iteration step $M$ at termination; learnt dictionary 
\begin{align} 
\dic^* &= \left[N^* \right]_\mathrm{SH} + \left[K^*\right]_\mathrm{APK},\\
N^* &=\{(n_i,j_i)\ |\ \textrm{there exists } i \in \{1,...,M\} \textrm{ such that } Y_{n_i,j_i} = d_i\},\\
K^* &= \{(r\xi)^{(i)}\ |\ \textrm{there exists } i \in \{1,...,M\} \textrm{ such that } P((r\xi)^{(i)},\cdot) = d_i\}
\end{align}
\end{compactitem}
\end{algo}

\subsection{Determination of candidates} \label{subsec:detercand}

For practice, the question remains how the candidates in each trial function class under consideration are determined. We need to explain what is done in S3.1 to S3.3.

First of all, we consider the determination of a spherical harmonic candidate. We want to seek the best-fitting function among all spherical harmonics up to a certain degree $\nu \in \nat$. We have to learn the size of $\nu$. It is defined in Algorithm \ref{algo:LRFMPmain} which specific spherical harmonics with a degree up to $\nu$ we insert into the learnt dictionary.

The idea is to allow the choice of spherical harmonics up to a degree $\widetilde{N} \in \nat$ (i.e. $\hat{N} = \{(n,j)\ |\ n \in \nat_0,\ n\leq\widetilde{N},\ j=-n,...,n\}$) which is probably not chosen in practice. 
For example, the data resolution can provide a threshold up to which a resolution appears to be realistic (as it is also done for other gravitational models like EGM or models based on CHAMP, GRACE and GOCE data).
If the LRFMP chooses only spherical harmonics with a lower degree $\nu$, we have a truly learnt bound $\nu < \widetilde{N}$. However, note that the higher we choose $\widetilde{N}$, the more expensive is the preprocessing of the algorithm. The candidate $d_{n+1}^{\mathrm{SH}}$ which the LRFMP chooses in each iteration step can be chosen as in the RFMP by comparing $\RFMP(Y_{m,k};n)$ for all spherical harmonics up to degree $\widetilde{N}$. 

Unfortunately, for meaningful results with respect to learning spherical harmonics, we are lacking appropriate data sets. Both the EGM2008 as well as the GRACE data are given to us only as coefficients for a representation in spherical harmonics. Thus, if we allow the LRFMP to choose an optimal spherical harmonic with an arbitrarily high degree, naturally, the algorithm tends to choose more spherical harmonics than other trial functions. If we were able to work with data not given in this polynomial representation, we would assume that this effect diminishes. 

Next, we consider the determination of the candidate $d_{n+1}^{\mathrm{APK}}$ from the Abel--Poisson kernels in S3.3. In this case, the minimization of the Tikhonov functional is modelled as a non-linear constrained optimization problem. The solution of this problem yields the respective candidate. Note that we do not seek the minimizer of a function, but the minimizer of a functional among a set of functions. Therefore, we have to define the trial functions dependent on their characteristics as we did in \eqref{eqAPk}. 

The model of the optimization problem is given as follows. Due to Theorem \ref{th:equivminmax}, in the $n$-th step of the LRFMP, we consider the optimization problem
\begin{align}
\RFMP (P(x,\cdot);n) \to \max! \label{eqminprobRFMP-1}
\end{align} 
for learning a dictionary for the RFMP. The optimization with respect to a trial function $d$ can be modelled by an optimization with respect to the characteristics of each trial function. However, these characteristics yield a constraint for the optimization problem.  Abel--Poisson kernels are given as $K_h(\xi \cdot) = P(h\xi, \cdot),\ h \in [0,1)$ and $\xi \in \Omega$ in \eqref{eqAPk}, see, for example, \citep[][p.~132]{Freedenetal2013_2} or \citep[][p.~52]{Freedenetal2004}. Here, $\xi$ is the centre of the radial basis function and $h$ is the parameter which controls the localization. Therefore, the kernels are well-defined only in the interior of the unit ball and the constraint is given by $\| x \|^2_{\real^3} < 1$ for $x=h\xi$.

\begin{defin}
The optimal candidate $d_{n+1}^{\mathrm{APK}}$ among the set of Abel--Poisson kernel $$\left\{P(x,\cdot)\ |\ x \in \mathring{\ball}_1(0)\right\}$$ is given by the solution of the optimization problem
\begin{align}
\RFMP(P(x,\cdot);n) \to \max ! \qquad s.t. \qquad \| x \|^2_{\real^3} < 1.
\end{align}
\end{defin}

Note that the maximizer is not necessarily unique. In this case, we use one representative among the maximizers.

We prefer a gradient-based approach. Thus, we have to compute the derivatives of \eqref{eqobjfunRFMP} with respect to $x \in \real^3$. In general, this can be done by applying the quotient rule and computing the derivatives of the inner products and norms in the de-/nominator separately. We state the results of this derivation at this point. A detailed derivation is given in Appendix \ref{app}.

\begin{theo} \label{th:abbrevs}
We define some abbreviation terms and state their derivatives. Let $R^n$ be the current residual of size $\ell$ and $F_n$ the current approximation, i.e. $F_n = \sum_{i=1}^n \alpha_i d_i$ for $d_i$ being a spherical harmonic $Y_{n_i,j_i}$ or an Abel--Poisson kernel $P\left(x^{(i)},\cdot\right)$. Moreover, let $\T_\daleth $ be the upward continuation operator and $\sigma$ the respective satellite orbit. Further, let the data be given on a point grid $\{\sigma\eta^{(i)}\}_{i=1,...,\ell},\ \eta_i \in \Omega$. We consider the Tikhonov functional with a penalty term dependent on the norm of the Sobolev space $\SBN_2$. At last, $P_n$ denotes a Legendre polynomial and $\era,\ \ephi,\ \ete$ represents the common local orthonormal basis vectors (up, East and North) in $\real^3$, see \eqref{eqONB-R3}. Then we have for an Abel--Poisson kernel $P(x,\cdot)$ and $x(r,\lon_x,t_x) \in \mathring{\ball}_1(0)$ the terms
\begin{align}
a_1(P(x,\cdot)) &\coloneqq \langle R^n, \T_\daleth P(x,\cdot) \rangle_{\real^\ell} 
	= \sum_{i=1}^\ell \frac{R^n_i}{4\pi} \frac{\sigma^2 - |x|^2}{\left(\sigma^2 + |x|^2 - 2\sigma x \cdot \eta^{(i)}\right)^{3/2}}, \label{eqa1} \\
a_2(P(x,\cdot)) &\coloneqq \langle F_n, P(x,\cdot) \rangle_{\SBN_2}	= \sum_{i=1}^n \alpha_i \left\{ \begin{matrix}
	T_1,  &d_i &= &Y_{n_i,j_i} \\
	T_2, &d_i &= &P\left(x^{(i)},\cdot\right),\\
	\end{matrix} \right. \label{eqa2}\\
T_1 &\coloneqq 	(n_i+0.5)^4 |x|^{n_i} Y_{n_i,j_i}\left( \frac{x}{|x|} \right)\\
T_2 &\coloneqq \sum_{n=0}^{\infty} (n+0.5)^4 \left( |x_i||x| \right)^n \frac{2n+1}{4\pi} P_n \left(\frac{x_i}{|x_i|} \cdot \frac{x}{|x|} \right)\\	
b_1(P(x,\cdot)) &\coloneqq \| \T_\daleth P(x,\cdot) \|^2_{\real^\ell}
	= \sum_{i=1}^\ell \frac{\left( \sigma^2 - |x|^2 \right)^2}{16\pi^2\left(\sigma^2 + |x|^2 - 2\sigma x \cdot \eta^{(i)}\right)^3}, \label{eqb1}\\
\intertext{and}
b_2(P(x,\cdot)) &\coloneqq \|  P(x,\cdot) \|^2_{\SBN_2}
	= \sum_{n=0}^\infty \frac{2n+1}{4\pi} (n+0.5)^4 |x|^{2n} . \label{eqb2}
\end{align} 
Their partial derivatives with respect to $x_j$ are given by 
\begin{align}
&\partial_{x_j} a_1(P(x,\cdot)) \\
&= - \sum_{i=1}^\ell \frac{R^n_i}{4\pi} \left[ \frac{2x_j}{\left(\sigma^2 + |x|^2 - 2\sigma x \cdot \eta^{(i)}\right)^{3/2}} + \frac{3\left(\sigma^2-|x|^2\right)\left(x_j-\sigma \eta_j^{(i)}\right)}{\left(\sigma^2 + |x|^2 - 2\sigma x \cdot \eta^{(i)}\right)^{5/2}} \right], \qquad \label{eqdxja1}
\end{align} 
\begin{align}
&\partial_{x_j} a_2(P(x,\cdot))= \sum_{i=1}^n \alpha_i \left\{ \begin{matrix}
	\partial_{x_j} T_1, &d_i &= &Y_{n_i,j_i} \\
	\partial_{x_j} T_2, &d_i &= &P\left(x^{(i)},\cdot\right),  \\
	\end{matrix} \right. \hspace*{0.5cm} \label{eqdxja2}\\
&\partial_{x_j} T_1 = (n_i+0.5)^4 \partial_{x_j} \left(|x|^{n_i} Y_{n_i,j_i}\left( \frac{x}{|x|} \right) \right)\\
&\partial_{x_j} T_2 = \sum_{n=0}^{\infty} (n+0.5)^4 \left( |x_i| \right)^n \frac{2n+1}{4\pi} \partial_{x_j} \left( |x|^n P_n \left(\frac{x_i}{|x_i|} \cdot \frac{x}{|x|} \right) \right)\\
&\partial_{x_j} b_1(P(x,\cdot)) \\
&= - \frac{1}{16\pi^2} \sum_{i=1}^\ell \left[ \frac{4x_j\left(\sigma^2-|x|^2\right)}{\left(\sigma^2 + |x|^2 - 2\sigma x \cdot \eta^{(i)}\right)^3} + \frac{6\left(\sigma^2-|x|^2\right)^2\left(x_j-\sigma\eta_j^{(i)}\right)}{\left(\sigma^2 + |x|^2 - 2\sigma x \cdot \eta^{(i)}\right)^4} \right], \qquad \label{eqdxjb1}
\intertext{and}
&\partial_{x_j} b_2(P(x,\cdot)) = \sum_{n=0}^\infty \frac{2n^2 + n}{2\pi} (n+0.5)^4 |x|^{2n-2}x_j. \label{eqdxjb2}
\end{align}
With respect to the derivative of $a_2$, we have further
\begin{align}
&\nabla \left(|x|^{n}Y_{n,j}\left( \frac{x}{|x|} \right) \right)\\
	&= \era n|x|^{n-1} Y_{n,j}\left( \frac{x}{|x|} \right)  + |x|^{n-1} \ephi \frac{j}{\sqrt{1-t^2}} Y_{n,-j}\left( \frac{x}{|x|} \right) \\ 
& + |x|^{n-1} \ete \sqrt{1-t_x^2} \sqrt{ \frac{2n+1}{4\pi} \frac{(n-|j|)!}{(n+|j|)!}} P'_{n,|j|}(t_x) 
\left\{\begin{matrix}
\sqrt{2} \cos(j\varphi_x), & j<0\\
1, & j=0 \\
\sqrt{2} \sin(j\varphi_x), & j>0\\
\end{matrix} \right. \label{eqnablaSH}
\intertext{and}
&\nabla \left(|x|^{n}P_n\left( \frac{x_i}{|x_i|} \cdot \frac{x}{|x|} \right)\right)\\
	&= |x|^{n-1} \left[ \era n P_n\left( \frac{x_i}{|x_i|} \cdot \frac{x}{|x|} \right) + \ephi P'_n\left( \frac{x_i}{|x_i|} \cdot \frac{x}{|x|} \right) \left( \frac{x_i}{|x_i|} \cdot \ephi \right) \right. \\
	& \qquad \qquad + \left.\ete  P'_n\left( \frac{x_i}{|x_i|} \cdot \frac{x}{|x|} \right) \left(\frac{x_i}{|x_i|} \cdot \ete \right) \right]. \label{eqnablaPn}
\end{align}
\end{theo}
For a proof, see Appendix \ref{app}.

\begin{theo} \label{th:gradRFMP}
With the abbreviations and derivatives from Theorem \ref{th:abbrevs}, the partial derivatives $\partial_{x_j},\ j=1,2,3,$ of $\RFMP(\cdot;n)$ are given by 
\begin{align}
\partial_{x_j} &\RFMP(P(x,\cdot);n) \\
&=\left[b_1(P(x,\cdot))+\lambda b_2(P(x,\cdot))\right]^{-2} \left[ 2[a_1(P(x,\cdot))-\lambda a_2(P(x,\cdot))] \right. \\
& \qquad \times [\partial_{x_j} a_1(P(x,\cdot))-\lambda \partial_{x_j}a_2(P(x,\cdot))] [b_1(P(x,\cdot))+\lambda b_2(P(x,\cdot))] \\
& \qquad \qquad  \left. - [a_1(P(x,\cdot))-\lambda a_2(P(x,\cdot))]^2[\partial_{x_j}b_1(P(x,\cdot))+\lambda \partial_{x_j}b_2(P(x,\cdot))] \right].
\end{align}
\end{theo}
\begin{proof} 
We only apply the common rules for derivatives. 
\end{proof}

Thus, for Abel--Poisson kernels, we determined the gradient of the modelled objective functions $\RFMP(\cdot;n),\ n \in \nat,$ analytically such that we are able to use a gradient-based optimization method. We use the primal-dual interior point filter line search method \textsc{Ipopt}. To enable parallelization, we installed the linear solver \texttt{ma97} from the HSL package. For more details on the \textsc{Ipopt} algorithm and the HSL package, see \citep{HSLRef,Nocedaletal2008,IpoptDocu,Waechteretal2005_1,Waechteretal2005_2,Waechteretal2006}. In practice, we set a few options manually which we explain in the necessary contexts later. In all other cases, we use the default option values. However, note that due to \citep{IpoptDocu}, we can only expect to obtain local solutions of the optimization problem. Furthermore, this algorithm uses a starting point to start its procedure towards an optimal Abel--Poisson kernel. This starting point is denoted by $d_{n+1}^{\mathrm{APK,loc}}$ in Algorithm \ref{algo:LRFMPmain} and is computed in S3.2. We obtain $d_{n+1}^{\mathrm{APK,loc}}$ by comparing the objective value $\RFMP(P(x,\cdot); n)$ for a selection of kernels given to the algorithm in the discrete dictionary by $[\hat{K}]_{\mathrm{APK}}$ and choosing the one with the highest value. As we have computed this kernel, we make use of it in S3.4 as well in case the optimization failed to find a better kernel. 

\subsection{Additional features for practice} \label{subsec:features}

The previously presented algorithm gives us a first and basic learning technique. During its development, we faced several problems dependent on the choice of the data and the given inverse problem. In order to overcome these difficulties, we introduced some additional features, which will be explained in the following.

First of all, the data of the monthly variation of the gravitational potential provided by GRACE attains very small values. In our experiments, these values lie in the interval [-0.1,0.1]. When inserting this data into the objective function for determining an optimal Abel--Poisson kernel, the \textsc{Ipopt} solver fails to find a solution at first. Thus, we set the option obj\_scaling\_factor to $-10^{10}$. For the EGM-data, we only use $-1$ to perform a maximization instead of a minimization. Note that the scaled objective is only seen internally to support the optimizer.

Next, to cut runtime and possible round-off errors, we implemented a restart method. We initiate a new run of the algorithm by resetting $F_E$ to zero after a previously chosen iteration number $E$. Note that, in contrast to the restart procedure of the ROFMP, see, for example, \citep{Telschow2014}, we also reset the regularization term $\lambda \| F_E \|_{\SBN_2}^2$ to zero. In our experiments, we used $E=250$. 

Furthermore, we saw that the learnt dictionary heavily depends on the regularization parameter. Thus, we have to use the same regularization parameter as we used during learning when applying the learnt dictionary. Moreover, in contrast to previous works on the IPMPs, see, for instance, \citep{Telschow2014}, the use of a non-stationary regularization parameter is necessary when learning a dictionary and applying it. In the previous works, the idea of a non-stationary, decreasing regularization parameter was explained in order to emphasize accuracy instead of smoothness of the obtained approximation. However, the improvement of the results did not justify the additional computational expenses of choosing a parameter and a decrease routine. Nonetheless, we reconsidered this idea since the main aim of the LRFMP is to learn a dictionary and a decreasing parameter appears to guide the learning process positively. Thus, we use a non-stationary regularization parameter $\lambda_n = (\lambda_0\| R_0\|_{\real^\ell}/(n+1))$, where $n$ is the current iteration number and $\lambda_0$ is an optimal regularization parameter. Our experiments show that with a decreasing regularization parameter, we determine a dictionary which yields a better approximation when applied. Hence, with a learnt dictionary, the use of a non-stationary regularization parameter has an impact on the result of the RFMP. 

Next, for the case of a high satellite orbit or the seasonal variation in the gravitational potential obtained via GRACE, we developed a dynamic dictionary approach. Note that in previous literature on dictionary learning, see, for example, \citep{Pruente2008}, it was mentioned that the structure of the input data could or even should be considered when learning a dictionary. We developed two strategies to learn a dynamic dictionary whose combination works well in the experiments under considerations. 

First of all, we demand that the first $250$ learnt dictionary elements are spherical harmonics. It seems that after these $250$ iterations, the current residual $R_n$ has a rougher structure than the initial residual $R_0$. Then the optimization routine finds more easily a more sensible solution and, therefore, learns a better dictionary. 
Additionally, we allow to only choose from the first $n+1$ learnt dictionary elements in the $n$-th iteration step of the RFMP when we use the learnt dictionary. In this case, the order of chosen dictionary elements from the LRFMP has to be preserved. In this way, the optimized trial function which was chosen in the $n$-th step of the LRFMP can be chosen in the RFMP as well. Additionally, we save runtime as the dictionary size is small at the beginning. At last, this also treats possible complications which might otherwise arise from the non-stationary regularization parameter. The fact that we decrease this parameter yields the choice of very localized trial functions in later iterations of the LRFMP. If we allowed the RFMP to choose them prematurely this would only lead to a reduced data error and not a better approximation as we have seen in practice. If we allow the RFMP to only choose from the first $n$ learnt dictionary elements, we can prevent it to choose very localized kernels prematurely.

With these features included in the LRFMP, we currently run our experiments.

\section{Numerical results} \label{sec:numerics}

\subsection{Setting of the experiments} \label{subsec:setting}

In this section, we present the results of the learnt dictionary compared to a manually chosen dictionary applied in the standard RFMP. We will use data from the EGM2008 as well as GRACE satellite data. The EGM2008 data is evaluated up to degree 1500. For the GRACE data, we computed the  meanfield from 2003 to 2013 averaged from the release 5 products from JPL, GFZ and CSR as was proposed in \citep{Sakumuraetal2014}. This meanfield was subtracted from the data corresponding to May 2008. Note that in May, traditionally, the wet season in the Amazon basin is about to end such that we can expect a concentration of masses in this region. Additionally, we smoothed the data with a Cubic Polynomial Scaling Function (see, for instance, \citep{Schreiner1996} and \citep[][p.~295]{Freedenetal1998}) of scale $5$. 

In all experiments, we compute the data on an equidistributed Reuter grid of 12684 data points, see, for instance, \citep{Reuter1982} and \citep[][p.~137]{Michel2013}. We choose the constant regularization parameter $\lambda$ of the RFMP and starting regularization parameter $\lambda_0$ of the LRFMP (see Subsection \ref{subsec:features}) such that they yield the lowest relative approximation error after 2000 iterations. In detail, we choose $\lambda=10^{-2}$ for both experiments and $\lambda_0 = 10^{-4}$ for the experiment with EGM data and $\lambda_0=10^{-1}$ for the experiment with GRACE data.

The arbitrarily chosen dictionary with which we compare the learnt dictionary is chosen similarly to \citep{Telschow2014}. We choose
\begin{align}
\dic^\mathrm{m} &\coloneqq [N^\mathrm{m}]_\mathrm{SH} + [K^\mathrm{m}]_\mathrm{APK}\label{eqdicraw}
\intertext{with}
N^\mathrm{m} &\coloneqq \{(n,j)\ |\ n=0,...,25;\ j=-n,...,n\},\\
K^\mathrm{m} &\coloneqq \{ r\xi\ |\ r \in R^\mathrm{m},\  \xi \in X^\mathrm{m}\}, \\
R^\mathrm{m} &\coloneqq  \{0.75,0.80,0.85,0.89,0.91,0.93,0.94,0.95,0.96,0.97\} 
\end{align}
and an equidistributed Reuter grid $X^\mathrm{m}$ of 4551 grid points on the sphere. All in all, the manually chosen dictionary contains $46186$ dictionary elements. 

The LRFMP needs a starting dictionary as well to provide the spherical harmonics and starting points for the optimization problem. For an equidistributed Reuter grid $X^\mathrm{s}$ of $1129$ grid points, we use 
\begin{align}
\dic^{\mathrm{s,EGM}} &\coloneqq \left[N^{\mathrm{s,EGM}}\right]_\mathrm{SH} + [K^\mathrm{s}]_\mathrm{APK}\label{eqdictra1}
\intertext{and}
\dic^{\mathrm{s,GRACE}} &\coloneqq \left[N^{\mathrm{s,GRACE}}\right]_\mathrm{SH} + [K^\mathrm{s}]_\mathrm{APK},\label{eqdictra2}
\intertext{respectively, with} 
N^{\mathrm{s,EGM}} &\coloneqq \{(n,j)\ |\ n=0,...,100;\ j=-n,...,n\},\\
N^{\mathrm{s,GRACE}} &\coloneqq \{(n,j)\ |\ n=0,...,60;\ j=-n,...,n\},\\
K^\mathrm{s} &\coloneqq \{ r\xi\ |\ r=0.94,\  \xi \in X^\mathrm{s}\},
\end{align}
From experience, we know that it is better to cut down the number of scales $r$ than the number of centres $\xi$ of the Abel--Poisson kernels. Thus, the starting dictionary contains $11330$ dictionary elements in the case of EGM2008 data and $4850$ dictionary elements in the case of GRACE data. 

We use the \textsc{Ipopt} optimizer with the linear solver \texttt{ma97} from HSL. For EGM data, we demand a desired and acceptable tolerance of $10^{-4}$ of the optimal solution. For GRACE data, we demand these tolerances to be only $10^0$ due to the scaling of the objective function explained in Section \ref{subsec:features}. As we need to ensure that the constraint is not violated during the optimization process, we set $r^2 <0.98999999^2$ as well as the options theta\_max\_fact $10^{-2}$ and watchdog\_shortened\_iter\_triggered $0$ in practice. For details on these options, see the \textsc{Ipopt} documentation \citep{IpoptDocu}.
 
We terminate the LRFMP as well as the RFMP in all cases when one of the following termination criteria is fulfilled. Either the relative data error is less than $10^{-8}$ or we reached $3000$ iterations. We also terminate the LRFMP when the last chosen coefficient $\alpha_{n+1}$ is less than $10^{-5}$. We experienced that otherwise the additionally learnt dictionary elements do not improve the solution and we can easily stop the learning process at this stage. Note that in stopping after at most 3000 iterations, we also limit the learnt dictionary to at most $3000$ learnt trial functions. This means, that the learnt dictionary is much smaller than the manually chosen dictionary with which we compare the learnt dictionary. Further, note that, due to its size, the manually chosen dictionary obviously has a much larger storage demand.

The results which we will compare here are obtained as follows: in one case, we use the RFMP with the manually chosen dictionary $\dic^{\mathrm{m}}$. In the second case, we first learn a dictionary $\dic^{\ast,\bullet}$ (see Algorithm \ref{algo:LRFMPmain}) for $\bullet \in \{ \mathrm{EGM}, \mathrm{GRACE} \}$ by using the LRFMP (which requires the starting dictionary $\dic^{\mathrm{s},\bullet}$) and then run the RFMP with the learnt dictionary $\dic^{\ast,\bullet}$. The major question is: is the learnt dictionary able to yield better results than the manually chosen dictionary?

The plots shown in this paper are done with MATLAB. Note that the colours for the results of the GRACE data are flipped in comparison to the results of the EGM data. This is done in order to emphasize wet regions with blue colour and dryer regions with red colour.

\subsection{Results} \label{subsec:results}

In Figure \ref{Fig:Results}, the results of the two experiments are shown. The first row shows the results with the EGM data. The second row depicts the results of the GRACE data. In the left-hand column the solution is given. The middle column shows the absolute approximation error of the RFMP with the manually chosen dictionary. The right-hand column depicts this error of the RFMP with the learnt dictionary. We adjusted the scales of the values for a better comparison.

Obviously, in both cases the algorithm is able to construct a good approximation. The relatively low errors occur basically within regions where more local structures are given. These regions are in the case of the EGM data in particular the Andean region as well as the Himalayas and the borders of the tectonic plates in Asia. In the case of monthly GRACE data, the masses in the Amazon basin originating from the ending wet season show the strongest structure. As we only allow 3000 iterations, it can be expected that such regions cannot be approximated perfectly. 

Particularly interesting are the results in the right-hand column which were obtained with the learnt dictionary. Clearly, in both scenarios, the approximation error is notably reduced. This is, in particular, also the case in the regions with localized anomalies.

\begin{sidewaysfigure}
	\includegraphics[width=0.33\textheight]{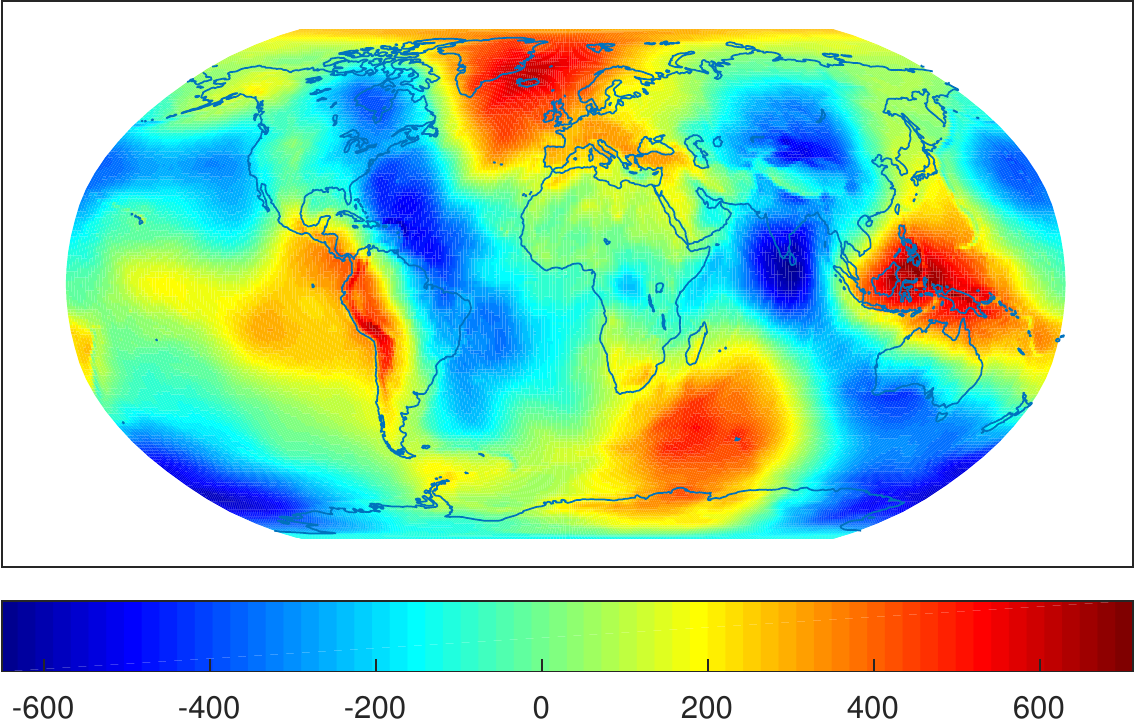}
	\includegraphics[width=0.33\textheight]{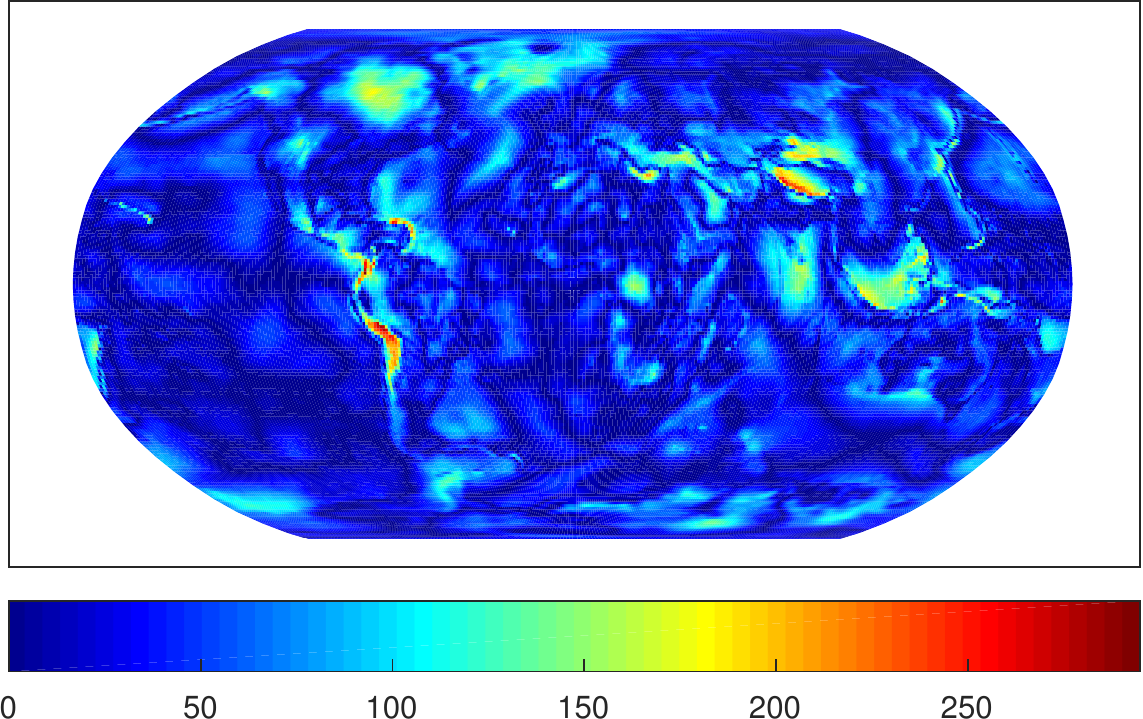}
	\includegraphics[width=0.33\textheight]{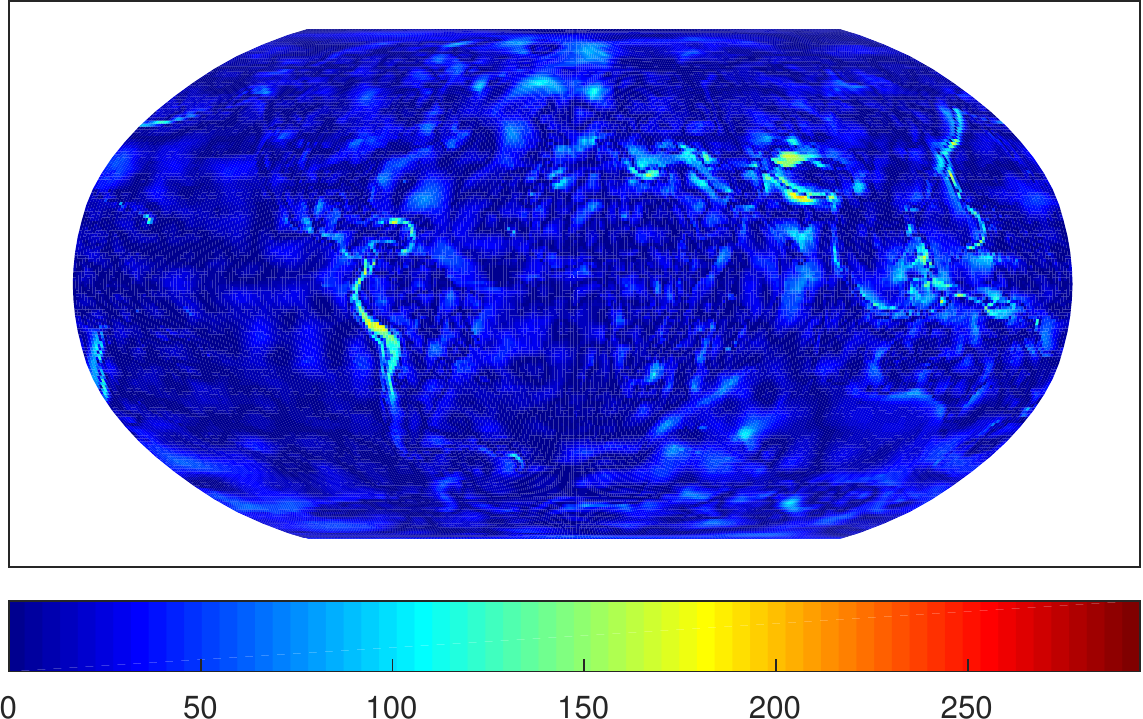}

	\includegraphics[width=0.33\textheight]{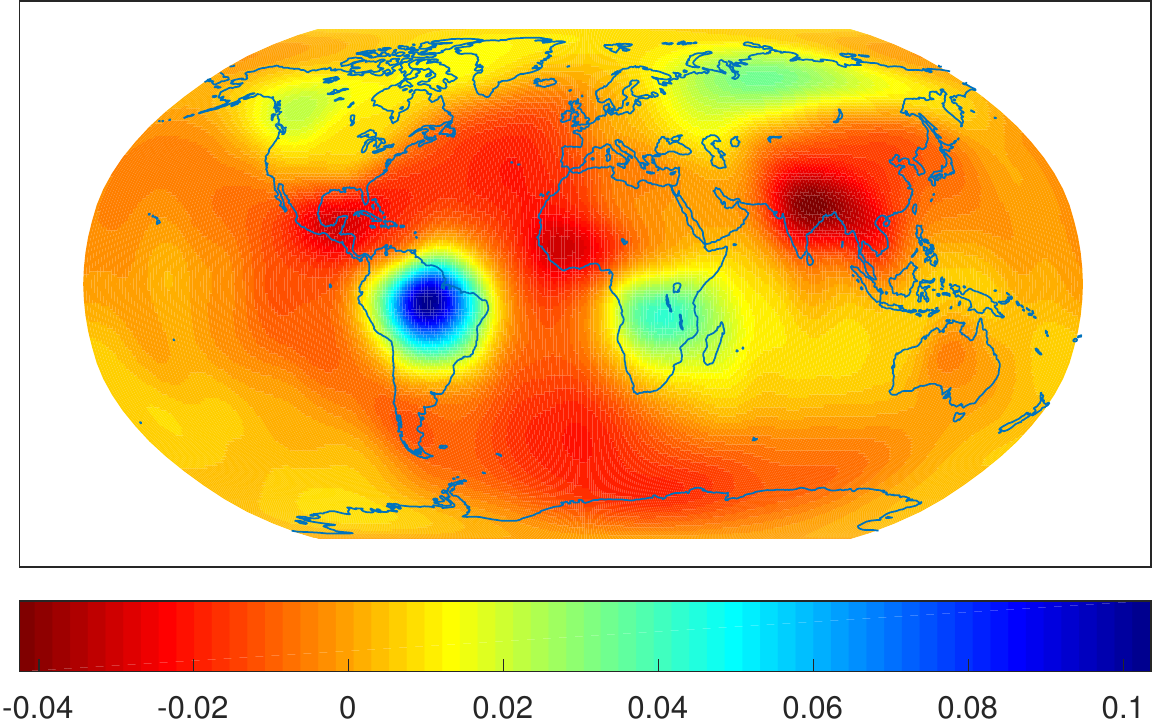}
	\includegraphics[width=0.33\textheight]{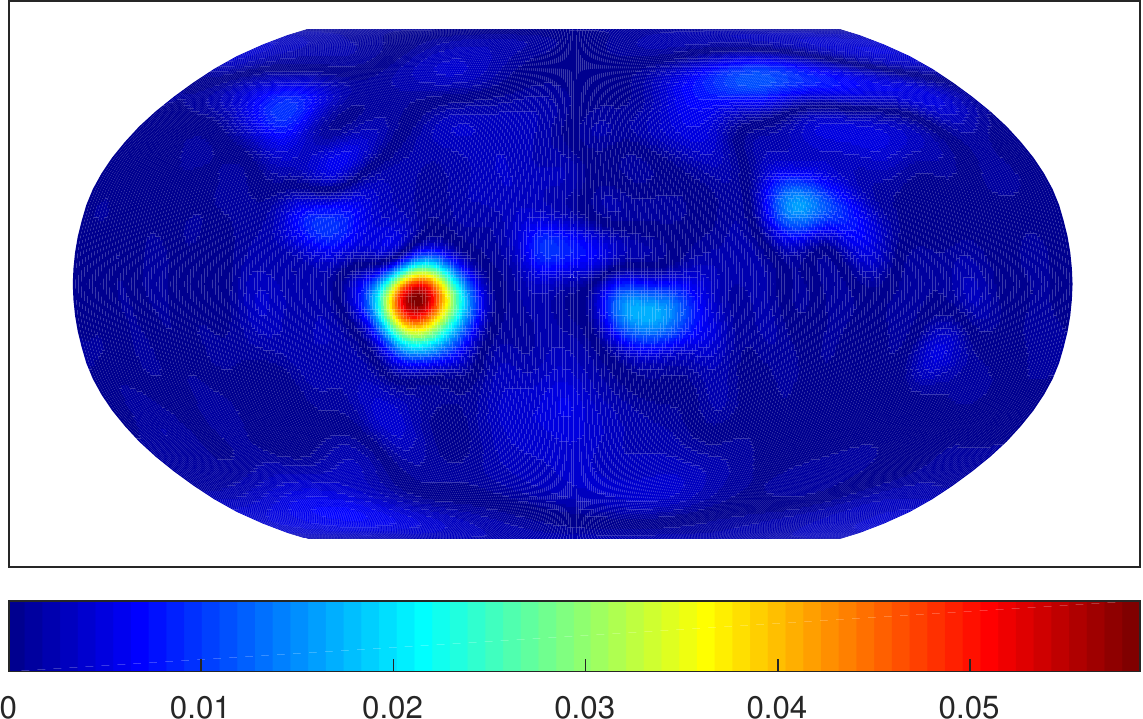}
	\includegraphics[width=0.33\textheight]{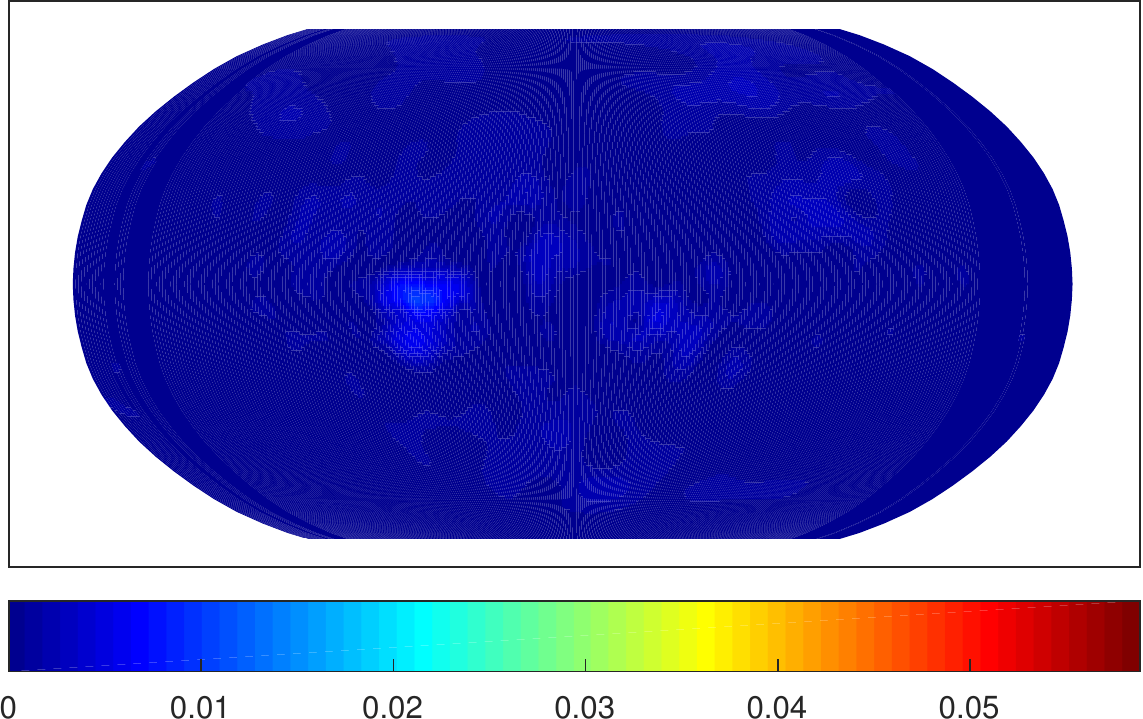}
	
	\caption{Results of the RFMP with a manually chosen dictionary and the learnt dictionary. Upper row: Results for EGM2008 data. Lower row: Results for GRACE data. Left: Solution. Middle: Absolute approximation error of RFMP with arbitrary dictionary. Right: Absolute approximation error of RFMP with learnt dictionary. 3000 iterations allowed in all experiments. All values in $\mathrm{m}^2/\mathrm{s}^2$.} \label{Fig:Results}	
\end{sidewaysfigure}

\begin{table}[bp]
\normalsize
\begin{tabular}{lcccc}
\hline
			&    			& rel. approximation  	& rel. data 	& CPU-runtime	\\ 
			&$\# \dic$		& error					& error			& in h		\\
\hline
RFMP$^*$ 	& $46186$ 	  	& 0.000794				& 0.065830		& 299.82\\
RFMP$^{**}$ & $\leq$ 3000 	& 0.000455				& 0.047092		& 41.34	\\ 
\hline
\hline
RFMP$^*$ 	& $46186$ 	  	& 0.001461				& 0.057765		& 295.49\\
RFMP$^{**}$ & $\leq$ 3000 	& 0.000306				& 0.046404		& 7.5	\\ 
\hline
\end{tabular}
\caption{Comparison of RFMP with a manually chosen dictionary (RFMP*) and a learnt dictionary (RFMP**). Upper comparison with respect to EGM2008 data. Lower comparison with respect to GRACE data. 3000 iterations allowed in all experiments.}
\label{Tab:Results} 
\end{table}


In Table \ref{Tab:Results}, the relative approximation and data errors after 3000 iterations of the RFMP with the manually chosen dictionary and the learnt dictionary are given. Furthermore, the (currently) needed CPU-runtime for the experiments are presented in the last column. We notice that we do not only obtain a smaller relative approximation error with the learnt dictionary but also a smaller relative data error. 

We state the CPU-runtime in hours for the sake of completeness. Although we see that we need less time to compute and preprocess the learnt dictionary than to preprocess the manually chosen dictionary, these results are to be understood with care because we currently do not work with an optimized code. 

All in all, the results show that we are able to learn a dictionary which yields a smaller data as well as approximation error than a manually chosen dictionary. In addition, we obtain these results with a sparser dictionary, less storage demand and an appropriate CPU-runtime.

\section{Conclusion and Outlook} \label{sec:conandout}

We started our investigations by aiming to improve our methods for gravity field modelling. In practice, we considered the RFMP for now. We expected to reduce the computational demands of the RFMP and the approximation error if a learnt dictionary is used rather than a manually chosen 'rule-of-thumb' dictionary.

In this paper, we presented a first approach to learn a dictionary of spherical harmonics and Abel--Poisson kernels for the downward continuation of gravitational data. In the numerical tests, we used data generated from EGM2008 and GRACE models. The idea of our learning approach is to iteratively minimize a Tikhonov functional over an infinite set of these trial functions. We do so by using non-linear constrained optimization techniques. Our results show that we obtain better results with respect to the relative data and approximation error when applying the learnt dictionary than when using a manually chosen dictionary in the RFMP. Moreover, we obtain these results with a sparser dictionary and less storage demand. Further, even non-optimized code yields satisfactory runtime results.

In the future research, we aim to transfer this learning approach to the ROFMP and enlarge its idea to Abel--Poisson wavelets and Slepian functions. Further, the presented learning algorithm can be viewed as one component of a structured learning technique. The question is whether we can determine an optimized learnt dictionary via learning from the results of the here presented strategy. In this way, we aim to build a learning hierarchy for the dictionary where the present technique represents the second level and the first level still needs to be developed. This could also lead to an automation of the additional features which we explained in Subsection \ref{subsec:features}. An additional objective is to obtain a dictionary for GRACE from a given set of training data to apply it with new test data such that we can provide an optimal dictionary for GRACE-FO satellite data. Further, we need to consider the theoretical aspects of the learning algorithm like determining a quantitative measure for the quality of a dictionary and investigating how the learnt dictionary of the LRFMP is related to that measure. In addition, with respect to practical aspects, we plan on optimizing our code to have more meaningful runtime results. 

\paragraph{Acknowledgement}
The authors gratefully acknowledge the financial support by the German Research Foundation (DFG; Deutsche Forschungsgemeinschaft), project MI 655/7-2.

\appendix

\section{Gradient of the objective function with respect to Abel--Poisson kernels} \label{app}
In this appendix, we prove Theorem \ref{th:abbrevs}. 

\paragraph{First considerations} We discuss the terms for the upward continuation operator as given in \eqref{eqpotential}. Note that the Euclidean inner product of two vectors is emphasized by using a '$\cdot$' at the particular positions. Additionally, we make use of the following basic aspects. 

In geomathematics, a common orthonormal basis in $\real^3$ is given by 
\begin{align}
\era(\lon,t) \coloneqq \left( \begin{matrix}
\sqrt{1-t^2} \cos(\lon) \\ \sqrt{1-t^2} \sin(\lon) \\ t
\end{matrix} \right),\
\ephi(\lon,t) \coloneqq \left( \begin{matrix}
-\sin(\lon) \\ \cos(\lon) \\ 0 
\end{matrix} \right),\
\ete(\lon,t) \coloneqq \left( \begin{matrix}
-t \cos(\lon) \\ -t \sin(\lon) \\ \sqrt{1-t^2}
\end{matrix} \right), \label{eqONB-R3}
\end{align}
see, for example, \citep[][p.~86]{Michel2013}. Note that it holds $\partial_\lon \era = \sqrt{1-t^2}\ephi$ and $\partial_t \era = \frac{1}{\sqrt{1-t^2}} \ete$. For the gradient $\nabla$, we will use a Cartesian definition as well as its decomposition into radial and angular parts. We have 
\begin{align}
\nabla = \left( \partial_{x_j} \right)_{j=1,2,3} 
= \era \frac{\partial}{\partial r} + \frac{1}{r} \left( \ephi \frac{1}{\sqrt{1-t^2}} \frac{\partial}{\partial \lon} + \ete \sqrt{1-t^2} \frac{\partial}{\partial t} \right), \label{eqnabla}
\end{align}
see, for instance, \citep[][p.~87]{Michel2013}. Next, we consider the following recurring inner products.
\begin{align}
\langle Y_{m,k}, Y_{n,j} \rangle_{\Lp{2}} &= \delta_{n,m}\delta_{j,k}, \label{eqIP1}\\
\langle P(x, \cdot), Y_{n,j} \rangle_{\Lp{2}} &= |x|^n Y_{n,j}\left( \frac{x}{|x|} \right), \label{eqIP2}
\end{align}
see, for instance, \citep[][p.~114]{Telschow2014}. At last, we note one specific property of the fully normalized spherical harmonics. It holds
\begin{align}
\frac{\partial}{\partial \lon} Y_{n,j}(\xi(\lon,t)) = jY_{n,-j}(\xi(\lon,t)), \label{eqSHprop}
\end{align}
see, for instance, \citep{Lewekeetal2018_2}. Now, we can compute the terms in Theorem \ref{th:abbrevs}.

\paragraph{The term $a_1(P(x,\cdot))$ and its derivative}
Obviously, for the formulation of $a_1(Px,\cdot))$ in \eqref{eqa1}, we only have to show that 
\begin{align}
\T P(x,\cdot) (\eta) = \frac{\sigma^2 - |x|^2}{4\pi(\sigma^2 + |x|^2 - 2\sigma x \cdot \eta)^{3/2}} \label{eqTP1}
\end{align}
for any $\eta \in \Omega$. We start at the left-hand side of \eqref{eqTP1}.
\begin{align}
\T P(x,\cdot)(\eta) 
&=\sum_{n=0}^\infty \sum_{j=1}^{2n+1} \langle Y_{n,j}, P(x,\cdot) \rangle_{\Lp{2}} \sigma^{-n-1} Y_{n,j}(\eta)\\
&=\sum_{n=0}^\infty \sum_{j=1}^{2n+1} \sigma^{-n-1} |x|^n Y_{n,j}\left( \frac{x}{|x|} \right)  Y_{n,j}(\eta)\\
&= \sum_{n=0}^\infty  |x|^n \sigma^{-n-1} \frac{2n+1}{4\pi} P_n\left( \frac{x}{|x|} \cdot \eta \right) 
= \frac{1}{\sigma} P \left(\frac{x}{\sigma},\eta\right) \\
&= \frac{1}{\sigma} \frac{1-|\tfrac{x}{\sigma}|^2}{4\pi(1+|\tfrac{x}{\sigma}|^2-2\tfrac{x}{\sigma}\cdot \eta)^{3/2}}  \\
&= \frac{\sigma^{-2} \left( \sigma^2-|x|^2 \right)}{4\pi\sigma \left( \sigma^{-2} \left( \sigma^2+|x|^2 - 2\sigma x \cdot \eta\right) \right)^{3/2} }   
= \frac{ \sigma^2-|x|^2 }{4\pi \left( \sigma^2+|x|^2 - 2\sigma x \cdot \eta \right)^{3/2} } . \label{eqTP}
\end{align}
Its derivative, as used in \eqref{eqdxja1}, is obtained via
\begin{align}
\partial_{x_j} &\T P(x,\cdot)(\eta)\\
&= \frac{\partial}{\partial x_j}\ \frac{ \sigma^2-|x|^2 }{4\pi \left( \sigma^2+|x|^2 - 2\sigma x \cdot \eta \right)^{3/2} }\\
&= \frac{ -2x_j\left(\sigma^2+|x|^2 - 2\sigma x \cdot\eta \right)^{3/2}  - 3\left(\sigma^2-|x|^2\right)\left(\sigma^2+|x|^2-2\sigma x \cdot \eta \right)^{1/2}\left( x_j - \sigma \eta_j \right)}{4\pi\left(\sigma^2 +|x|^2-2\sigma x \cdot \eta \right)^3} \\
&= -\frac{1}{4\pi}  \left[ \frac{ 2x_j} {\left(\sigma^2+|x|^2-2\sigma x \cdot \eta \right)^{3/2}} + \frac{ 3\left(\sigma^2-|x|^2\right)\left( x_j - \sigma\eta_j \right)  } {\left(\sigma^2+|x|^2-2\sigma x \cdot \eta \right)^{5/2}} \right].
\label{eqderivTP}
\end{align}

\paragraph{The term $a_2(P(x,\cdot))$ and its derivative}
For the current approximation $F_n$, we write $F_n = \sum_{i=1}^n \alpha_i d_i$ for dictionary elements $d_i = Y_{n_i,j_i}$ or $d_i=P\left(x^{(i)},\cdot\right)$ depending on what element was chosen in the $i$-th step. Then we can derive the representation \eqref{eqa2} of $a_2(P(x,\cdot))$ as follows 
\begin{align}
a_2(P(x,\cdot)) &\coloneqq \langle F_n, P(x,\cdot) \rangle_{\SBN_2} = \sum_{i=1}^n \alpha_i \langle d_i, P(x,\cdot) \rangle_{\SBN_2}\\
	& = \sum_{i=1}^n \alpha_i \left\{ \begin{matrix}
	(n_i+0.5)^4 |x|^{n_i} Y_{n_i,j_i}\left( \frac{x}{|x|} \right), &d_i &= &Y_{n_i,j_i}\\
	\sum_{n=0}^{\infty} (n+0.5)^4 \left( |x_i||x| \right)^n \frac{2n+1}{4\pi} P_n \left(\frac{x_i}{|x_i|} \cdot \frac{x}{|x|} \right), &d_i &= &P\left(x^{(i)},\cdot\right) \\
	\end{matrix} \right. \\
\end{align}
due to \eqref{eqIP1}, \eqref{eqIP2} and the addition theorem for spherical harmonics. The derivative of $a_2(P(x,\cdot))$ as given in \eqref{eqdxja2} is obvious. However, we have to show \eqref{eqnablaSH} and \eqref{eqnablaPn}. We will exchange the term $|x|$ by $r$ as well as $\frac{x}{|x|}$ by $\xi$ and $\frac{x^{(i)}}{|x^{(i)}|}$ by $\xi^{(i)}$ for this. Then we obtain the following results. We first consider \eqref{eqnablaSH}.

\begin{align}
\nabla_x &\left( r^{n}Y_{n,j}(\xi) \right) \notag\\
&= \era \left( \pdervr r^{n} Y_{n,j}(\xi(\lon,t)) \right) + \frac{1}{r} \ephi \frac{1}{\sqrt{1-t^2}} \left( r^{n} \pdervlon  Y_{n,j}(\xi(\lon,t)) \right)\\ 
& \qquad + \frac{1}{r} \ete \sqrt{1-t^2} \left( r^{n} \pdervt  Y_{n,j}(\xi(\lon,t)) \right) \\
&= \era nr^{n-1} Y_{n,j}(\xi(\lon,t))  + r^{n-1} \ephi \frac{j}{\sqrt{1-t^2}} Y_{n,-j}(\xi(\lon,t)) \notag \\ 
& \qquad + r^{n-1} \ete \sqrt{1-t^2} \sqrt{ \frac{2n+1}{4\pi} \frac{(n-|j|)!}{(n+|j|)!}} P'_{n,|j|}(t) \left\{ \begin{matrix}
\sqrt{2} \cos(j\lon), &j<0,\\
1, &j=0,\\
\sqrt{2} \sin(j\lon),  &j>0,\\
\end{matrix} \right. \label{eqnabla_h_SH}
\end{align}
where we used \eqref{eqnabla}, \eqref{eqSHprop} and \eqref{eqfnsh}. We have to take a closer look at \eqref{eqnabla_h_SH} regarding a possible singularity in $t=\pm 1$. The term \eqref{eqnabla_h_SH} contains two possibly problematic terms:
\begin{align}
\frac{j}{\sqrt{1-t^2}} Y_{n,-j}(\xi) \qquad \textrm{ and } \qquad 
\sqrt{1-t^2} \sqrt{ \frac{2n+1}{4\pi} \frac{(n-|j|)!}{(n+|j|)!}} P'_{n,|j|}(t). \label{eqSingterms}
\end{align}
We first consider the term on the left-hand side of \eqref{eqSingterms}. Obviously, if $j=0$, this is a removable singularity with a zero value. In the case $j\not=0$, we recall the definition of the fully normalized spherical harmonics, which we use in practice, from \eqref{eqfnsh} and \eqref{eqaLf}. Thus, for the problematic term, we have
\begin{align}
\frac{j}{\sqrt{1-t^2}} &Y_{n,-j}(\xi) \\
&= j \sqrt{ \frac{2n+1}{4\pi} \frac{(n-|j|)!}{(n+|j|)!}} \left(1-t^2\right)^{(|j|-1)/2} \left( \frac{\mathrm{d}^{|j|}}{\mathrm{d}t^{|j|}} P_n(t)\right) \left\{ \begin{matrix}
\sqrt{2} \cos(j\lon), & j<0,\\
1, & j=0,\\
\sqrt{2} \sin(j\lon), & j>0.\\
\end{matrix} \right. \label{eqSingSHFREE}
\end{align}
Obviously, there exists a problem only for $\tfrac{|j|-1}{2} < 0 \Leftrightarrow |j|-1<0 \Leftrightarrow |j|<1 \Leftrightarrow j=0.$ However, this is excluded in this case. Thus, there is no problem in the term of the left-hand side of \eqref{eqSingterms}. With respect to the term on the right-hand side of \eqref{eqSingterms}, we have
\begin{align}
&\sqrt{1-t^2} \sqrt{ \frac{2n+1}{4\pi} \frac{(n-|j|)!}{(n+|j|)!}} P'_{n,|j|}(t)\\
&= \sqrt{1-t^2} \sqrt{ \frac{2n+1}{4\pi} \frac{(n-|j|)!}{(n+|j|)!}} \frac{\mathrm{d}}{\mathrm{d}t} \left[ \left(1-t^2\right)^{|j|/2} \frac{\mathrm{d}^{|j|}}{\mathrm{d}t^{|j|}} P_n(t) \right] \notag \\
&= \sqrt{ \frac{2n+1}{4\pi} \frac{(n-|j|)!}{(n+|j|)!}} \left[ -jt \left(1-t^2\right)^{(|j|-1)/2} \frac{\mathrm{d}^{|j|}}{\mathrm{d}t^{|j|}} P_n(t) + \left(1-t^2\right)^{(|j|+1)/2} \frac{\mathrm{d}^{|j|+1}}{\mathrm{d}t^{|j|+1}} P_n(t) \right] \label{eqDerPnjFREE}
\end{align}
with the definition in \eqref{eqaLf}. Obviously, the problematic term is 
\begin{align}
\tfrac{j}{2} \left(1-t^2\right)^{(|j|-1)/2}.
\end{align}
If $j=0$, the term vanishes. If $j\not=0$, the exponents are non-negative. Thus, also the term on the right-hand side of \eqref{eqSingterms} contains no singularity. 

At last, for the derivative of $a_2(P(x,\cdot))$, we need to consider \eqref{eqnablaPn}. We obtain
\begin{align}
\nabla_x &\left( r^{n}  P_n(\xi^{(i)} \cdot \xi) \right) = \nabla_x \left(r^{n}  P_n(\xi^{(i)} \cdot \era) \right)\\
& = \era nr^{n-1}  P_n(\xi^{(i)} \cdot \era) + r^{n-1} \left( \ephi \frac{1}{\sqrt{1-t^2}} P'_n(\xi^{(i)} \cdot \era)  \left( \xi^{(i)} \cdot \pdervlon \era \right) \right.  \\
& \qquad \left. +\ \ete \sqrt{1-t^2}  P'_n(\xi^{(i)} \cdot \era) \left(\xi^{(i)} \cdot \pdervt \era\right) \right)  \notag\\
& = \era nr^{n-1}  P_n(\xi^{(i)} \cdot \era) + r^{n-1} \left( \ephi \frac{1}{\sqrt{1-t^2}} P'_n(\xi^{(i)} \cdot \era)  \left( \xi^{(i)} \cdot \sqrt{1-t^2}\ephi \right) \right.  \\
& \qquad \left. +\ \ete \sqrt{1-t^2}  P'_n(\xi^{(i)} \cdot \era) \left(\xi^{(i)} \cdot \frac{1}{\sqrt{1-t^2}}\ete \right) \right)  \\
& = r^{n-1} \left( \era n P_n(\xi^{(i)} \cdot \era) + \ephi P'_n(\xi^{(i)} \cdot \era)  \left( \xi^{(i)} \cdot \ephi \right) + \ete  P'_n(\xi^{(i)} \cdot \era) \left(\xi^{(i)} \cdot \ete \right) \right). \label{eqnabla_h_Pn}
\end{align}
With $\era = \frac{x}{|x|}$, this is the formulation of \eqref{eqnablaPn}.

\paragraph{The term $b_1(P(x,\cdot))$ and its derivative}
The term $b_1(P(x,\cdot))$ as in \eqref{eqb1} is obvious when we take \eqref{eqTP} into account. The derivative of $b_1(P(x,\cdot))$ as in \eqref{eqdxjb1} is obtained by
\begin{align}
\partial_{x_j} &b_1 (P( x, \cdot))\\
&= \frac{\partial}{\partial x_j } \frac{1}{16\pi^2}\sum_{i=0}^\ell \frac{ \left( \sigma^2-|x|^2 \right)^2}{\left( \sigma^2+|x|^2 - 2\sigma x \cdot \eta^{(i)} \right)^3 }\\
&=  \frac{1}{16\pi^2}\sum_{i=0}^\ell \left( \left( \sigma^2+|x|^2 -2\sigma x\cdot\eta^{(i)} \right)^{-6} \left(
-4x_j\left( \sigma^2-|x|^2 \right) \left( \sigma^2+|x|^2-2\sigma x\cdot\eta^{(i)} \right)^3 \right. \right. \\
& \left. \left. \qquad \qquad \qquad \qquad - 6\left( \sigma^2-|x|^2 \right)^2 \left( \sigma^2+|x|^2-2\sigma x\cdot\eta^{(i)} \right)^2 \left( x_j - \sigma \eta_j^{(i)} \right) \right) \right)\\
&=  -\frac{1}{16\pi^2}\sum_{i=0}^\ell \left(\frac{ 4x_j\left( \sigma^2-|x|^2 \right) }{ \left( \sigma^2+|x|^2 -2 \sigma x\cdot\eta^{(i)} \right)^3} + \frac{6\left( \sigma^2 -|x|^2 \right)^2 \left( x_j - \sigma \eta_j^{(i)} \right)}{ \left( \sigma^2+|x|^2 -2 \sigma x\cdot\eta^{(i)} \right)^4 } \right). \label{Pdxjb1}
\end{align}

\paragraph{The term $b_2(P(x,\cdot))$ and its derivative}

The formulation as in \eqref{eqb2} of $b_2(P(x,\cdot))$ is due to the following considerations which use \eqref{eqIP2} and $P_n(1) = 1$, see, for instance, \cite[][p.~49]{Michel2013}.
\begin{align}
\|P(x,\cdot)\|_{\SBN_2}^2 
	&= \sum_{n=0}^\infty \sum_{j=1}^{2n+1} (n+0.5)^4 \langle P(x, \cdot), Y_{n,j} \rangle_{\Lp{2}}^2\\
	&= \sum_{n=0}^\infty \sum_{j=1}^{2n+1} (n+0.5)^4 |x|^{2n} \left( Y_{n,j} \left( \frac{x}{|x|} \right) \right)^2\\
	&= \sum_{n=0}^\infty (n+0.5)^4 |x|^{2n} \frac{2n+1}{4\pi} P_n \left(\frac{x}{|x|} \cdot \frac{x}{|x|} \right)
	= \sum_{n=0}^\infty \frac{2n+1}{4\pi} (n+0.5)^4 |x|^{2n} .
\end{align}
The gradient with respect to $x$ is then obtained as follows.
\begin{align}
\nabla_x b_2(P(x,\cdot))
=  \sum_{n=0}^\infty  (n+0.5)^4 \frac{2n+1}{4\pi} \nabla_x \left( |x|^{2n} \right). \label{eqdxjPb2_1}
\end{align} 
With \eqref{eqnabla}, we have
\begin{align}
\nabla_x \left(  |x|^{2n} \right)
= \nabla_x r^{2n}
&= 2n \era r^{2n-1}
= 2n \era |x|^{2n-1}. \label{eqnabla_r_2n}
\end{align} 
Inserting this result into \eqref{eqdxjPb2_1}, we obtain
\begin{align}
\nabla_x b_2 (P(x,\cdot))
&=  \sum_{n=0}^\infty   \frac{2n^2+n}{2\pi} (n+0.5)^4\era |x|^{2n-1} 
=  \sum_{n=0}^\infty   \frac{2n^2+n}{2\pi} (n+0.5)^4 x |x|^{2n-2}. \label{eqdxjPb2}
\end{align} 
This is in accordance with \eqref{eqdxjb2}. Hence, Theorem \ref{th:abbrevs} is proven. \qed

\bibliographystyle{abbrvnat}     
{\footnotesize
\bibliography{biblio}}

\end{document}